\documentclass{amsart}
\usepackage{graphicx}        
\usepackage{multicol}        
\usepackage[bottom]{footmisc}

\usepackage{rotating}

\theoremstyle{plain}
\newtheorem{theorem}{Theorem}[section]
\newtheorem{lemma}[theorem]{Lemma}
\newtheorem{corollary}[theorem]{Corollary}

\newtheorem{conjecture}[theorem]{Conjecture}
\newtheorem{proposition}[theorem]{Proposition}

\theoremstyle{remark}
\newtheorem{remark}[theorem]{Remark}

\theoremstyle{definition}
\newtheorem{definition}[theorem]{Definition}
\newtheorem{assumption}[theorem]{Assumption}

\usepackage{amssymb}
\usepackage{amsmath}
\usepackage{xspace}
\usepackage{enumerate}
\usepackage[all]{xy}
\usepackage{pifont} 

\newcommand{\si}[1]{\mbox{\textsf{#1}}}

\newcommand{\pisy}{\blacksquare}
\newcommand{\afsy}{\square}
\newcommand{\Xpi}{X_\pisy}
\newcommand{\Xaf}{X_\afsy}

\newcommand{\tempre}{\tau}
\newcommand{\busby}{\eta}
\newcommand{\A}{\ensuremath{\mathfrak{A}}\xspace}
\newcommand{\B}{\ensuremath{\mathfrak{B}}\xspace}
\newcommand{\IdealI}{\ensuremath{\mathfrak{I}}\xspace}
\newcommand{\IdealJ}{\ensuremath{\mathfrak{J}}\xspace}

\newcommand{\FK}[2]{\ensuremath{\mathsf{FK}_{#1}\left(#2\right)}\xspace}
\newcommand{\FKplus}[2]{\ensuremath{\mathsf{FK}^+_{#1}\left(#2\right)}\xspace}

\newcommand{\tref}[1]{Theorem \ref{#1}}
\newcommand{\pref}[1]{Proposition~\ref{#1}}
\newcommand{\cref}[1]{Corollary~\ref{#1}}
\newcommand{\rref}[1]{Remark~\ref{#1}}

	\newcommand{\kl}{\ensuremath{\mathit{KL}}\xspace}
	\newcommand{\kk}{\ensuremath{\mathit{KK}}\xspace}

	\newcommand{\multialg}[1]{\mathcal{M}(#1)\xspace}

\newcommand{\bK}{{\mathbb K}}
\newcommand{\X}{{\mathsf X}}

\newcommand{\K}{\mathbb{K}}

		\newcommand{\setof}[2]{\ensuremath{\left\{ #1 \: : \: #2 \right\}}}
\newcommand{\Ca}{$C^*$-al\-ge\-bra\xspace}

\newcommand{\shom}{$*$-ho\-mo\-mor\-phism\xspace}
\newcommand{\AFa}{$AF$ al\-ge\-bra\xspace}
\newcommand{\Cas}{$C^*$-al\-ge\-bras\xspace}

\newcommand{\shoms}{$*$-ho\-mo\-mor\-phisms\xspace}
\newcommand{\AFas}{$AF$ al\-ge\-bras\xspace}
\newcommand{\Oia}{$\mathcal{O}_\infty$-absorbing\xspace}

\DeclareMathOperator{\Prim}{Prim}

\DeclareMathOperator{\Ad}{Ad}

\newcommand{\toep}[1][E]{\mathcal T(E)}

	\newcommand{\ftn}[3]{ #1 : #2 \rightarrow #3 }
	\newcommand{\corona}[1]{\mathcal{Q}(#1)\xspace}

\newcommand{\unital}{\mathbf 1}
\newcommand{\fg}{f.g.}

\newcommand{\AF}{$\square$}
\newcommand{\PI}{$\blacksquare$}
\newcommand{\AFPI}{\ding{111}}
\usepackage{ifthen}

\usepackage{tikz}
\usetikzlibrary[arrows,snakes,backgrounds]

\usepackage{intcalc}

\newcommand{\oooI}[1]{\ifthenelse{\equal{#1}{gray}}{\AFPI}{\ifthenelse{\isodd{#1}}{\PI}{\AF}}}
\newcommand{\ooIo}[1]{\ifthenelse{\equal{#1}{gray}}{\AFPI}{\ifthenelse{\isodd{\intcalcDiv{#1}{2}}}{\PI}{\AF}}}
\newcommand{\oIoo}[1]{\ifthenelse{\equal{#1}{gray}}{\AFPI}{\ifthenelse{\isodd{\intcalcDiv{#1}{4}}}{\PI}{\AF}}}
\newcommand{\Iooo}[1]{\ifthenelse{\equal{#1}{gray}}{\AFPI}{\ifthenelse{\isodd{\intcalcDiv{#1}{8}}}{\PI}{\AF}}}

\newcommand{\twop}[1]{
\begin{tikzpicture}
\node at ( 0,0) {\ooIo{#1}};
\node at ( 1,0) {\oooI{#1}};
\draw [->]  (0.2,0) -- (0.8,0);
\end{tikzpicture}}

\newcommand{\threeplin}[1]{
\begin{tikzpicture}
\node at ( 0,0) {\oIoo{#1}};
\node at ( 1,0) {\ooIo{#1}};
\node at ( 2,0) {\oooI{#1}};
\draw [->]  (0.2,0) -- (0.8,0);
\draw [->]  (1.2,0) -- (1.8,0);
\end{tikzpicture}}

\newcommand{\threepin}[1]{
\begin{tikzpicture}
\node at ( 0,0) {\oIoo{#1}};
\node at ( 1,0) {\ooIo{#1}};
\node at ( 2,0) {\oooI{#1}};
\draw [->]  (0.2,0) -- (0.8,0);
\draw [->]  (1.8,0) -- (1.2,0);
\end{tikzpicture}}

\newcommand{\threepout}[1]{
\begin{tikzpicture}
\node at ( 0,0) {\oIoo{#1}};
\node at ( 1,0) {\ooIo{#1}};
\node at ( 2,0) {\oooI{#1}};
\draw [->]  (0.8,0) -- (0.2,0);
\draw [->]  (1.2,0) -- (1.8,0);
\end{tikzpicture}}

\newcommand{\fourA}[1]{
\begin{tikzpicture}
\node at ( 0,0.7) {\Iooo{#1}};
\node at ( 1,0.7) {\oIoo{#1}};
\node at ( 2,0.7) {\ooIo{#1}};
\node at ( 1,0) {\oooI{#1}};
\draw [->]  (0,0.5) -- (0.8,0.2);
\draw [->]  (1,0.5) -- (1,0.2);
\draw [->]  (2,0.5) -- (1.2,0.2);
\end{tikzpicture}}

\newcommand{\fourE}[1]{
\begin{tikzpicture}
\node at ( 0,0) {\Iooo{#1}};
\node at ( 1,0) {\oooI{#1}};
\node at ( 2,0) {\ooIo{#1}};
\node at ( 3,0) {\oIoo{#1}};
\draw [->]  (0.2,0) -- (0.8,0);
\draw [->]  (2.2,0) -- (2.8,0);
\draw [->]  (1.8,0) -- (1.2,0);
\end{tikzpicture}}
\newcommand{\fourF}[1]{
\begin{tikzpicture}
\node at ( 0,0) {\Iooo{#1}};
\node at ( 1,0) {\oooI{#1}};
\node at ( 2,0) {\ooIo{#1}};
\node at ( 3,0) {\oIoo{#1}};
\draw [->]  (0.2,0) -- (0.8,0);
\draw [->]  (1.8,0) -- (1.2,0);
\draw [->]  (2.8,0) -- (2.2,0);
\end{tikzpicture}}
\newcommand{\fouroneE}[1]{
\begin{tikzpicture}
\node at ( 0.7,0) {\oIoo{#1}};
\node at ( 0.7,0.7) {\ooIo{#1}};
\node at ( 0,0) {\oooI{#1}};
\node at ( 0,0.7) {\Iooo{#1}};
\draw [->]  (0.7,0.2) -- (0.7,0.5);
\draw [->]  (0.5,0) -- (0.2,0);
\draw [->]  (0.2,0.7) -- (0.5,0.7);
\draw [->]  (0,0.5) -- (0,0.2);
\end{tikzpicture}}
\newcommand{\fouroneF}[1]{
\begin{tikzpicture}
\node at ( 0.7,0) {\oooI{#1}};
\node at ( 0.7,0.7) {\oIoo{#1}};
\node at ( 0,0) {\ooIo{#1}};
\node at ( 0,0.7) {\Iooo{#1}};
\draw [->]  (0.5,0.5) -- (0.2,0.2);
\draw [->]  (0.2,0) -- (0.5,0);
\draw [->]  (0,0.5) -- (0,0.2);
\end{tikzpicture}}

\newcommand{\fourthreeeight}[1]{
\begin{tikzpicture}
\node at ( 0,0) {\oIoo{#1}};
\node at ( 1,0) {\ooIo{#1}};
\node at ( 2,0) {\oooI{#1}};
\node at ( 1,0.7) {\Iooo{#1}};
\draw [->]  (0.8,0.5) -- (0.2,0.2);
\draw [->]  (1,0.5) -- (1,0.2);
\draw [->]  (1.2,0.5) -- (1.8,0.2);
\end{tikzpicture}}

\newcommand{\fourthreenine}[1]{
\begin{tikzpicture}
\node at ( 0,0) {\oIoo{#1}};
\node at ( 1,0) {\Iooo{#1}};
\node at ( 2,0) {\ooIo{#1}};
\node at ( 3,0) {\oooI{#1}};
\draw [->]  (0.8,0) -- (0.2,0);
\draw [->]  (1.2,0) -- (1.8,0);
\draw [->]  (2.2,0) -- (2.8,0);
\end{tikzpicture}}

\newcommand{\fourthreeB}[1]{
\begin{tikzpicture}
\node at ( 0.7,0) {\oooI{#1}};
\node at ( 0.7,0.7) {\oIoo{#1}};
\node at ( 0,0) {\ooIo{#1}};
\node at ( 0,0.7) {\Iooo{#1}};
\draw [->]  (0.7,0.5) -- (0.7,0.2);
\draw [->]  (0.2,0) -- (0.5,0);
\draw [->]  (0.2,0.7) -- (0.5,0.7);
\draw [->]  (0,0.5) -- (0,0.2);
\end{tikzpicture}}

\newcommand{\fourthreeE}[1]{
\begin{tikzpicture}
\node at ( 0.7,0) {\oooI{#1}};
\node at ( 0.7,0.7) {\oIoo{#1}};
\node at ( 0,0) {\ooIo{#1}};
\node at ( 0,0.7) {\Iooo{#1}};
\draw [->]  (0.2,0.7) -- (0.5,0.7);
\draw [->]  (0.5,0.5) -- (0.2,0.2);
\draw [->]  (0.7,0.5) -- (0.7,0.2);
\end{tikzpicture}}

\newcommand{\fourthreeF}[1]{
\begin{tikzpicture}
\node at ( 0,0) {\Iooo{#1}};
\node at ( 1,0) {\oIoo{#1}};
\node at ( 2,0) {\ooIo{#1}};
\node at ( 3,0) {\oooI{#1}};
\draw [->]  (0.2,0) -- (0.8,0);
\draw [->]  (1.2,0) -- (1.8,0);
\draw [->]  (2.2,0) -- (2.8,0);
\end{tikzpicture}}

\title[Graph $C^*$-algebras with no more than four
primitive ideals]{Classification of graph $C^*$-algebras with no more than four primitive ideals}

\author{S{\o}ren Eilers}
\author{Gunnar Restorff}
\author{Efren Ruiz}

\address{Department of Mathematical Sciences\\
        University of Copenhagen\\
        Universitetsparken~5\\
        DK-2100 Copenhagen\\Denmark}
        \email{eilers@math.ku.dk}
\address{
Department of Science and Technology\\
University of the Faroe Islands\\ N\'oat\'un 3\\FO-100 T\'orshavn\\Faroe Islands}
\email{gunnarr@setur.fo}
\address{
Department of Mathematics\\University of Hawaii\\
Hilo, 200 W. Kawili St.\\
Hilo, Hawaii\\
96720-4091 USA}
        \email{ruize@hawaii.edu}

\newcommand{\hex}[1]{\ifthenelse{#1<10}{#1}{\ifthenelse{#1=10}{A}{\ifthenelse{#1=11}{B}{\ifthenelse{#1=12}{C}{\ifthenelse{#1=13}{D}{\ifthenelse{#1=14}{E}{F}}}}}}}

\newcommand{\msp}[2]{\textsf{#1.#2}}
\newcommand{\myref}[3]{\textsf{#1.#2.\hex{#3}}}
\newcommand{\mylab}[3]{\textsf{\hex{#3}}\label{l#1_#2_#3}}
\newcommand{\mspk}[2]{\textsf{#1.#2.x}}
\newcommand{\AFc}{\tref{AF}}
\newcommand{\PIc}{\tref{PI}}

\newcommand{\FININFPIideal}{\pref{prop:classgraph1}}
\newcommand{\FININFAFideal}{\pref{prop:classgraph2}}
\newcommand{\ETPIideal}{\pref{prop:classgraph1}}
\newcommand{\ETAFideal}{\pref{prop:classgraph2}}

\newcommand{\fandownAFideal}{\tref{t:class2}}
\newcommand{\fandownPIideal}{\tref{t:class2-PI}}

\newcommand{\fanupAFquotient}{\tref{t:class1}}
\newcommand{\fanupPIquotient}{\tref{t:class1-PI}}

\date{\today}
\begin{document}
\begin{abstract}
We describe the status quo of the classification problem of graph $C^*$-algebras with four primitive ideals or less.
\end{abstract}
\maketitle

\section{Introduction}

The class of graph $C^*$-algebras (cf.~\cite{ir:ga} and the references therein) has proven to be an important and interesting venue for classification theory by $K$-theoretical invariants; in particular with respect to $C^*$-algebras with finitely many ideals, and in 2009, the authors formulated the following \emph{working conjecture}:

\begin{conjecture}\label{woco}
  Graph \Cas $C^*(E)$ with finitely many ideals are classified up to stable isomorphism by their filtered, ordered $K$-theory $\FKplus{\operatorname{Prim}(C^*(E))}{C^*(E)}$.
\end{conjecture}

Here, the filtered, ordered $K$-theory is simply the collection of all $K_0$- and $K_1$-groups of subquotients of the $C^*$-algebra in question, taking into account all the natural transformations among them (details will be given below). The conjecture addresses the possibility of a classification result which is not {strong} (cf.~\cite{gae:ttc}) in the sense that we do not expect every possible isomorphism at the level of the invariant to lift to the $C^*$-algebras.

The conjecture remains open and we are forthwith optimistic about its veracity, although some of the results which have been obtained, as we shall see, seem to indicate that an added condition of finitely generated $K$-theory could be needed. In the present paper we will discuss the status of this conjecture for graph algebras with four or fewer primitive ideals; if the number is three or fewer we can present a complete classification under the condition of finitely generated $K$-theory, but for the number four there are many cases still eluding our methods. Adding, in some cases, the condition of finitely generated $K$-theory -- or even stronger, that the graph algebra is unital -- we may solve 103 of the 125 cases, leaving less than one fifth of the cases open.
Our main contribution in the present paper concerns the class of \emph{fan spaces} which has not been accessible through the methods we have used earlier, but we will also go through those results in our two papers \cite{segrer:ccfis} and \cite{segrer:cecc} which apply here.


\subsection{Tempered primitive ideal spaces}

Invoking an idea from \cite{seerapws:agc} we organize our overview using a \emph{tempered ideal space} of the \Ca in question. This is defined for any \Ca with only finitely many ideals as the pair $(\Prim(\A),\tempre)$ where $\tempre:\Prim(\A)\to \{0,1\}$ is defined as
\[
\tempre(\IdealI)=\begin{cases}0&K_0(\IdealI/\IdealI_0)_+\not=K_0(\IdealI/\IdealI_0)\\1&K_0(\IdealI/\IdealI_0)_+=K_0(\IdealI/\IdealI_0)
\end{cases}
\]
with $\IdealI_0$ the maximal proper ideal of $\IdealI$ (this exists by the fact that $\IdealI$ is prime and contains only finitely many ideals).
We set
\[
\Xaf=\{x\in X\mid \tau(x)=0\}\qquad \Xpi=\{x\in X\mid \tau(x)=1\}
\]
To be able to work systematically with these objects, we now give them a combinatorial description.

\begin{definition} 
Let \A be a \Ca. 
We let $\Prim(\A)$ denote the \emph{primitive ideal space} of \A, equipped with the usual hull-kernel topology, also called the Jacobson topology. 
We always identify the open sets of $\Prim(\A)$, $\mathbb{O}(\Prim(\A))$, and the lattice of ideals of \A, $\mathbb{I}(\A)$, using the lattice isomorphism
$$U\mapsto \bigcap_{\mathfrak{p}\in \Prim(\A)\setminus U}\mathfrak{p} . $$ When $U$ is an open set we write $\A(U)$ for the corresponding ideal of $\A$. When $U\supset V$ are both open, so that $U\setminus V$ is locally closed, we write $\A(U\setminus V)$ for the subquotient $\A(U)/\A(V)$.
\end{definition}

Note that whenever $\Xaf$ or $\Xpi$ are locally closed, standard results in graph $C^*$-algebra theory give that $\A(\Xaf)$ and $\A(\Xpi)$ are \AFas and \Oia algebras, respectively.

\begin{definition}
Let $X$ be a topological space. The \emph{specialisation preorder} $\prec$ on $X$ is defined by $x\prec y$ if and only if $x\in\overline{\{y\}}$.
\end{definition}

A topological space satisfies the $T_0$ separation axiom if and only if its specialisation preorder is a partial order. 

\begin{definition}
A subset $H$ of a preordered set $(X,\leq)$ is called \emph{hereditary} if $x\leq y\in H$ implies $x\in H$. 
\end{definition}

\begin{definition}
Let $(X,\leq)$ be a preordered set. 
The \emph{Alexandrov topology} of $X$ is the topology with the closed sets being the hereditary sets. 

A topological set is called an \emph{Alexandrov space} if it carries the Alexandrov topology of some preordered set. The preorder is necessarily the specialisation preorder. 
A topological space is an Alexandrov space if and only if arbitrary intersections of open sets are open. 
\end{definition}

Since we are dealing with \Cas with finite primitive ideal spaces, these are all Alexandrov spaces satisfying the $T_0$ separation axiom. 
Consequently, we can equivalently consider all partial orders on finite sets. The tempered primitive ideal space for a \Ca with $n$ primitive ideals may hence be uniquely described using a partial order on $\{1,\dots,n\}$ and a map in $\{0,1\}^{\{1,\dots,n\}}$.

The \emph{transitive reduction} of a relation $R$ on a set $X$ is a minimal relation $S$ on $X$ having the same transitive closure as $R$. In general neither existence nor uniqueness are guaranteed, but if the transitive closure of $R$ is antisymmetric and finite, there is a unique transitive reduction. 
We will illustrate our (finite) topological spaces with graphs of the transitive reduction of the specialisation order, where we write an arrow $x\rightarrow y$ if and only if $x$ is less than $y$ in the transitive reduction of the specialisation order (similar to the Hasse diagram). The value of $\tempre$ will be indicated by colors of the vertices of the graph; white for $0$ and black for $1$.

We obtain a unique signature for each tempered ideal space as follows. Consider the adjacency matrix of the graph of the specialisation order and recall that (by transitivity and antisymmetry) we can always permute the vertices so that the adjacency matrix becomes an upper triangular matrix. Since the relation is reflexive, we will have ones in the diagonal, so without loss of information we may write the values of $\tempre$ there.
To each such upper triangular matrix 
\[
A=\begin{bmatrix} t_1&a_{1,2}&&a_{1,n-1}&a_{1,n}\\&t_2&a_{2,3}&&a_{2,n}\\
&&\ddots&\ddots&\\
&&&t_{n-1}&a_{n-1,n}\\
&&&&t_n
\end{bmatrix}
\]
 we associate two binary numbers
\[
a=a_{1,2}a_{1,3}\cdots a_{1,n}a_{2,3}a_{2,4}\cdots a_{2,n}\cdots a_{n-1,n}
\]
and 
\[
t=t_1\cdots t_n
\]
In general, there are several such binary numbers associated with a specialisation order by means of permuting the vertices. We choose the order of the vertices to obtain the smallest possible pair $(a,t)$ ordered lexicographically as the unique identifier for this specific tempered ideal structure. In the interest of conserving space we write hexadecimal expansion of the numbers when referring to a certain structure. We write \textsf{n.a.t} and \textsf{n.a} to indicate \emph{signatures} and \emph{tempered signatures}, respectively, defined this way (where \textsf{n} and \textsf{a} are numbers written in decimal expansions and \textsf{t} is a number written in hexadecimal expansion).

If a primitive ideal space is disconnected, we may classify the \Cas associated to each component individually. We will hence assume throughout that the \Cas have connected primitive ideal space (when considering graph algebras, a necessary, but not sufficient, condition for this is that the underlying graphs are connected considered as undirected graphs). Determining the number of connected $T_0$-spaces with $n$ points is hard for most $n$; the number has been computed up to $n=16$ in \cite{gbbdm:psp}. But for small $n$ even the number of tempered ideal spaces can readily be found by naive enumeration, by first counting all spaces and then performing inverse Euler transform to obtain those that are connected:

\begin{center}
\begin{tabular}{|l||c|c|c|c|c|c|}\hline
$|\Prim(\A)|$&1&2&3&4&5&6\\\hline \hline
Number of spaces&1&2&5&16&63&318\\\hline
Number of connected spaces&1&1&3&10&44&238\\\hline\hline
Number of tempered spaces&2&10&62&510&5292&69364\\\hline
Number of connected tempered spaces&2&4&20&125&1058&11549\\\hline
\end{tabular}
\end{center}

We will restrict our attention to $|\Prim(\A)|\leq 4$ and hence have 15 (connected) primitive ideal spaces\footnote{The space \msp{4}{E} was forgotten on page 230 of \cite{rmrn:ctsbc}}  which may be given temperatures in a total of 151 different ways to concern ourselves with:
\begin{center}
\begin{tabular}{|c|c|c|}\hline
\msp{1}{0}&\AFPI&[L],[A] \\\hline
\msp{2}{1}&\twop{gray}&[L],[A]\\\hline
\msp{3}{7}&\threeplin{gray}&[L],[A]\\\hline 
\msp{4}{E}&\fourE{gray}&[A]\\\hline
\msp{4}{F}&\fourF{gray}&[A]\\\hline
\msp{4}{39}&\fourthreenine{gray}&[A]\\\hline
\msp{4}{3F}&\fourthreeF{gray}&[L],[A]\\\hline
                  \end{tabular}
\qquad
\begin{tabular}{|c|c|c|}\hline
\msp{3}{3}&\threepin{gray}&[A],[F]\\\hline
\msp{3}{6}&\threepout{gray}&[A],[F]\\\hline
\msp{4}{A}&\fourA{gray}&[F]\\\hline  
\msp{4}{38}&\fourthreeeight{gray}&[F]\\\hline
\end{tabular}
\quad             
\begin{tabular}{|c|c|c|}\hline
\msp{4}{1F}&\fouroneF{gray}&[Y]\\\hline
\msp{4}{3E}&\fourthreeE{gray}&[Y]\\\hline
\msp{4}{1E}&\fouroneE{gray}&[O]\\\hline
\msp{4}{3B}&\fourthreeB{gray}&[O]\\\hline
\end{tabular}
\end{center}
where \AFPI{} just indicates that it is either \AF{} or \PI. 

We call a finite $T_0$ space \emph{linear} ([L]) if its partial order is total. Following \cite{rbmk:uctcfts} we call it an \emph{accordion space} ([A]) if the symmetrization of the space is the symmetrization of a linear space. We call it a \emph{fan space} ([F]) when there is a smallest or largest element in the preorder, so that when this is removed, what remains is a disjoint union of linear spaces. The remaining spaces we organize as [Y]-spaces and [O]-spaces as indicated.
 In Section~\ref{resultsstart} below we summarize our results subject to this organization.

\subsection{The invariant}

Let $\mathfrak A$ be a $C^*$-algebra with finitely many ideals and set $\mathsf X=\operatorname{Prim}(\mathfrak A)$. Note that for any locally closed subset $Y = U \setminus V$ of $X$, we have two groups $K_{0} ( \A (Y) )$ and $K_{1} ( \A ( Y ) )$.  Moreover, for any three open subsets  $U \subseteq  V \subseteq W$ of $X$, we have a six term exact sequence
\[
\xymatrix{
{K_0(\A ( Y_{1} ) ) }\ar[r]^{\iota_0}&{K_0(\A(Y_{2}) )}\ar[r]^-{\pi_0}&{K_0(\A( Y_{3})) }\ar[d]^{\partial_{0}}\\
{K_1(\A ( Y_{3} )) }\ar[u]^-{\partial_{1}}&{K_1(\A( Y_{2} ) )}\ar[l]^-{\pi_1}&{K_1(\A( Y_{1}) )}\ar[l]^-{\iota_1}}
\]
where $Y_{1} = V \setminus U$, $Y_{2} = W \setminus U$, and $Y_{3} = W \setminus V$.  The \emph{filtered, ordered $K$-theory}%
$\mathrm{FK}^+_{\mathsf X}(\mathfrak A)$ of $\mathfrak A$ is the collection of all $K$-groups thus occurring, equipped with order on $K_0$ and the natural transformations $\{\iota_*,\pi_*,\partial_{*}\}$.

Consequently, if also $\operatorname{Prim}(\mathfrak B)=\mathsf X$, we write $\mathrm{FK}^+_{\mathsf X}(\mathfrak A)\cong \mathrm{FK}^+_{\mathsf X}(\mathfrak B)$ if for each locally closed subset $Y$ of $X$, there exist group isomorphisms
\[
\alpha^{Y}_*:K_*(\A(Y))\to K_*(B( Y) )
\]
preserving all natural transformations in such a way that all 
$\alpha^{Y}_0$ are also order isomorphisms.
All components of this invariant are readily computable (\cite{tmcsemt:imkga}), and often, much of it is redundant. We will not pursue that issue here.

The \emph{filtered $K$-theory}%
$\mathrm{FK}_{\mathsf X}(\mathfrak A)$ of $\mathfrak A$ is defined analogously by disregarding the order structure on $K_0$. The filtered (ordered) $K$-theory over a finite $T_0$-space $X$ can also be used for \Cas over $X$ without being tight.\footnote{Although this is not exactly the same definition as the filtrated $K$-theory in \cite{rmrn:ctsfk}, it is known to be the same for all the cases where we have a UCT. For more on this invariant and \Cas over $X$ the reader is referred to \cite{rmrn:ctsfk} and the references therein. } 

%

\subsection{Graph \Ca}

A \emph{graph} $(E^0, E^1, r, s)$ consists of a countable set $E^0$ of vertices, a countable set $E^1$ of edges, and maps $r : E^1 \to E^0$ and $s : E^1 \to E^0$ identifying the range and source of each edge.  If $E$ is a graph, the \emph{graph $C^*$-algebra} $C^*(E)$ is the universal
\Ca generated by mutually orthogonal projections $\{ p_v
: v \in E^0 \}$ and partial isometries $\{ s_e : e \in E^1 \}$ with mutually orthogonal
ranges satisfying
\begin{enumerate}
\item $s_e^* s_e = p_{r(e)}$ \quad  for all $e \in E^1$
\item $s_es_e^* \leq p_{s(e)}$ \quad for all $e \in E^1$
\item $p_v = \sum_{\{ e \in E^1 : s(e) = v \}} s_es_e^* $ \quad for all $v$ with $0 < | s^{-1}(v) | < \infty$.
\end{enumerate}

The countability hypothesis ensures that all our graph \Cas are separable, which is a necessary hypothesis for many of the classification results.  We will be mainly interested in graph \Cas with real rank zero.  For a graph $E$, we have that the real rank of $C^{*} (E)$ is zero if and only if $E$ is satisfying \emph{Condition (K)}, i.e., no vertex of $E$ is the base point of exactly one simple cycle (see Theorem~3.5 of \cite{jj:realrankzero}).  Moreover, by Proposition~3.3 of \cite{jj:realrankzero}, every graph \Ca with finitely many ideals has real rank zero. 
Thus, every graph $C^{*}$-algebra with finitely many ideals has a norm-full projection, and by 
\cite{lgb:sihsc}, every graph $C^{*}$-algebra with finitely many ideals is stably isomorphic to a unital $C^{*}$-algebra.

Throughout the paper we will use the following facts about graph $C^{*}$-algebras without further mention.

\begin{theorem}
Let $C^{*} (E)$ be a unital graph $C^{*}$-algebra satisfying Condition (K).  

\begin{enumerate}[(1)] 
\item Every ideal of $C^{*} (E)$ is stably isomorphic to a unital graph $C^{*}$-algebra. 

\item Every sub-quotient of $C^{*} (E)$ is stably isomorphic to a unital graph $C^{*}$-algebra.

\item The $K$-groups of every sub-quotient of $C^{*} (E)$ is finitely generated.

\item  Every non-unital simple sub-quotient of $C^{*} (E)$ that is an AF-algebra is isomorphic to $\K$.
\end{enumerate}
\end{theorem}

\begin{proof}
As in the proof of Theorem~5.7 (4) of \cite{mt:ideals} (see also \cite[Proposition~3.4]{bhrs:idealstrucgraph}), every ideal of a graph $C^{*}$-algebra satisfying Condition (K) is Morita equivalent to $C^{*}(F)$, where $F^{0} \subseteq E^{0}$.  Hence, (1) holds since a graph $C^{*}$-algebra $C^{*} (E)$ is unital if and only if $E^{0}$ is finite.  (2) follows from (1) and \cite[Corollary~3.5]{bhrs:idealstrucgraph}.  (3) follows from (2) and \cite[Theorem~3.1]{ddmt:kthy}.

Suppose $C^{*} (F)$ is a simple unital AF-algebra.  Then $F$ has no cycles.  Since $C^{*} (F)$ is unital, $F^{0}$ is finite.  Therefore, $F$ has a sink.  By \cite[Corollary~2.15]{ddmt:cag}, every singular vertex must be reached by any other vertex since $C^{*} (F)$ is simple.  Thus, $F$ must be a finite graph.  Hence, $C^{*} (F) \cong \mathsf{M}_{n}$.  From this observation, (4) follows from (1) and (2) since any non-unital simple $C^{*}$-algebra stably isomorphic to $\K$ is isomorphic to $\K$.
\end{proof}

See \cite{ir:ga} and the references therein for more on graph \Cas.

\section{General theory}
We first describe the situations in which the graph algebras can be classified using widely applicable results.

\subsection{The \textit{AF} case}
The \textit{AF} case corresponds to temperatures that are constantly 0. We incur these at the tempered signatures 
\myref{1}{0}{0},
\myref{2}{1}{0},
\myref{3}{3}{0},
\myref{3}{6}{0},
\myref{3}{7}{0},
\myref{4}{A}{0},
\myref{4}{E}{0},
\myref{4}{F}{0},
\myref{4}{1E}{0},
\myref{4}{1F}{0},
\myref{4}{38}{0},
\myref{4}{39}{0},
\myref{4}{3B}{0},
\myref{4}{3E}{0}, and
\myref{4}{3F}{0}. Of course the classification question is resolved by Elliott's theorem:

\begin{theorem}[\cite{gae:cilssfa}]\label{AF}
\AFas are classified up to stable isomorphism by their ordered $K_0$-group.
\end{theorem}

\subsection{The purely infinite case}
Recall that there are three notions of pure infiniteness for non-simple \Cas, namely \emph{pure infiniteness}, \emph{strong pure infiniteness}, and \emph{$\mathcal{O}_\infty$-absorption}, introduced by E. Kirchberg and M. R\o{}rdam; cf.~\cite{ekmr:purelyinf} and~\cite{ekmr:oinftyabs}. 

\begin{corollary}
For each nuclear, separable \Ca \A with finite primitive ideal space, the following are equivalent: 
\begin{enumerate}[(a)]
\item
\A is purely infinite,
\item
\A is strongly purely infinite,
\item 
\A is \Oia, i.e., $\A\otimes\mathcal{O}_\infty\cong\A$.
\end{enumerate}
\end{corollary}
\begin{proof}
It follows from Theorem~9.1 and Corollary~9.2 of \cite{ekmr:oinftyabs} that (c) implies (b), that (b) implies (a), and that the three coincide in the simple case. It follows from Proposition~3.5 of \cite{ekmr:oinftyabs}, that pure infiniteness passes to ideals and subquotients. 
Thus it follows from \cite{atww:selfabs} that (a) implies (c). 
\end{proof}

The purely infinite case (the \Oia case) corresponds to temperatures that are constantly 1. We incur these at the tempered signatures 
\myref{1}{0}{1},
\myref{2}{1}{3},
\myref{3}{3}{7},
\myref{3}{6}{7},
\myref{3}{7}{7},
\myref{4}{A}{15},
\myref{4}{E}{15},
\myref{4}{F}{15},
\myref{4}{1E}{15},
\myref{4}{1F}{15},
\myref{4}{38}{15},
\myref{4}{39}{15},
\myref{4}{3B}{15},
\myref{4}{3E}{15}, and
\myref{4}{3F}{15}. As we will outline below, all but the case \myref{4}{1E}{15} are resolved through the recent work of many hands.

The isomorphism result of Kirchberg (cf.~\cite{ek:cpicukt} and~\cite{ek:nkmkna}) reduces the classification problem of nuclear and strongly purely infinite \Cas which are also in the bootstrap class to an isomorphism problem in ideal-related $\kk$-theory. Since all purely infinite graph \Cas fall in this class we may hence confirm Conjecture \ref{woco} in the purely infinite case by providing a universal coefficient theorem which allows the lifting of isomorphisms at the level of filtered $K$-theory to invertible $\kk_X$-classes. This, however, is not known to be possible in general. Indeed, Meyer and Nest in \cite{rmrn:ctsfk} showed that there are purely infinite \Cas over the space \msp{4}{A} which fails to have this property, but since the examples provided there cannot possibly come from graph algebras, the question remains open in general. The work of Bentmann and K\"ohler established that general UCTs are available precisely when the space $X$ is an accordion space, and Arklint with the second and third named authors provided UCTs for other spaces, including \msp{4}{A}, under the added assumption that the \Ca has real rank zero which is automatic here. Specializing even further, Arklint, Bentmann and Katsura provided a UCT which applies for our space \msp{4}{3B} under the added assumption that the \Ca has real rank zero and that the $K_1$ groups of all subquotients are free, which also is automatic here. The space \msp{4}{1E} remains open. 
In conclusion:

\begin{theorem}\label{PI}
Purely infinite, separable, nuclear \Cas \A with finite primitive ideal space $X$ in the bootstrap class of Meyer and Nest (i.e., all simple subquotients are in the bootstrap class of Rosenberg and Schochet) are classified up to stable isomorphism by their filtered $K$-theory $\FK{X}{-}$ in the cases
\begin{enumerate}[(i)]
\item $X$ is an accordion space [\msp{1}{0}, \msp{2}{1}, \msp{3}{3}, \msp{3}{6}, \msp{3}{7}, \msp{4}{E}, \msp{4}{F}, \msp{4}{39}, \msp{4}{3F}] (\cite{ek:cpicukt}, \cite{ncp:pureinf}, \cite{mk:ext}, \cite{gr:ckalg}, \cite{rmrn:ctsfk}, \cite{rbmk:uctcfts}, \cite{ek:nkmkna})
\item $X$ is one of the spaces \msp{4}{A}, \msp{4}{38}, \msp{4}{1F}, \msp{4}{3E} and $\operatorname{rr}(\A)=0$. (\cite{sagrer:fkrrzc})
\item $X$ is the space \msp{4}{3B}, $\operatorname{rr}(\A)=0$, and $K_1(\IdealJ/\IdealI)$ is free for any  $\IdealI\triangleleft \IdealJ \unlhd \A$ (\cite{sarbtk:rfkccka})
\end{enumerate}
\end{theorem}

\subsection{The separated case}\label{sep}
The classification problem for the two \emph{mixed} cases with $|\Prim(\A)|=2$ not covered by the results mentioned above -- the tempered signatures \myref{2}{1}{1} and \myref{2}{1}{2} -- were resolved in \cite{semt:cnga} drawing heavily on \cite{segrer:cecc}. In \cite{segrer:ccfis}, we generalized this to more complicated cases having the separation property which is automatic in the two-point case, as detailed below. The idea is to find an ideal $\IdealI$ such that $\IdealI$ is \textit{AF} and $\A/\IdealI$ is \Oia, or vice versa. We do not know in general how to prove classification in this case, but under certain added assumptions related to the notion of fullness, this leads to results that may be used to resolve the cases of tempered signature \myref{3}{7}{1}, \myref{3}{7}{3}, \myref{4}{F}{1}, \myref{4}{1F}{1}, \myref{4}{1F}{3}, \myref{4}{3B}{1}, \myref{4}{3F}{1}, \myref{4}{3F}{3}, \myref{4}{3F}{7} by Proposition \ref{prop:classgraph1} below and \myref{3}{7}{4}, \myref{3}{7}{6}, \myref{4}{39}{8}, \myref{4}{3B}{8}, \myref{4}{3E}{8}, \myref{4}{3E}{12}, \myref{4}{3F}{8}, \myref{4}{3F}{12}, \myref{4}{3F}{14} by Proposition \ref{prop:classgraph2}. 


\begin{definition}
Let $n>1$ be a given integer. 
Then we let $\X_n$ denote the partially ordered set (actually totally ordered) $\X_n=\{1,2,\ldots,n\}$ with the usual order. 
For $a,b\in \X_n$ with $a\leq b$, we let $[a,b]$ denote the set $\setof{x\in \X_n}{a\leq x \leq b}$.
\end{definition}

\begin{proposition}\label{prop:classgraph1}
Let $\A_{1}$ and $\A_{2}$ be separable, nuclear, \Cas over $\X_{n}$ in the bootstrap class of Meyer and Nest (i.e., every simple subquotient is in the bootstrap class of Rosenberg and Schochet).  Suppose $\A_{i}(\{ 1\})$ is an \AFa and $\A_{i}([2, n])$ is a tight stable $\mathcal{O}_{\infty}$-absorbing \Ca over $[2, n]$, and $\A_{i}(\{ 2 \})$ is an essential ideal of $\mathfrak{ A }_{i}([1 , 2 ])$.  Then $\A_{1} \otimes \K \cong \A_{2} \otimes \K$ if and only if there exists an isomorphism $\ftn{ \alpha }{ \FK{\X_{n}}{\A_{1}} } { \FK{\X_{n}}{\A_{2}} }$ such that $\alpha_{ \{ 1 \} }$ is positive. 
\end{proposition}

\begin{proposition}\label{prop:classgraph2}
Let $\A_{1}$ and $\A_{2}$ be graph \Cas satisfying Condition (K).  Suppose $\A_{i}$ is a \Ca over $\X_{n}$ such that $\A_{i}(\{ n \})$ is an \AFa, for every ideal \IdealI of $\A_{i}$ we have that $\IdealI \subseteq \A_{i}(\{ n \})$ or $\A_{i}(\{ n \}) \subseteq \IdealI$, and $\A_{i}([1, n-1])$ is a tight, $\mathcal{O}_{ \infty }$-absorbing \Ca over $[1,n-1]$.  Then $\A_{1} \otimes \K \cong \A_{2} \otimes \K$ if and only if there exists an isomorphism $\ftn{ \alpha }{ \FK{\X_{n}}{\A_{1}} } { \FK{\X_{n}}{\A_{2}} }$ such that $\alpha_{ \{ n \} }$ is positive. 
\end{proposition}

\section{Fan spaces}
\label{sec:fanspaces}

In this section, we develop methods to deal mainly with the spaces \msp{3}{3}, \msp{3}{6}, \msp{4}{A}, \msp{4}{38}.
We observe the following in \cite{segrer:ccfis}

\begin{lemma}\label{cfpbasics}
Let $E$ be a graph such that $C^*(E)$ has finitely many ideals and assume that $\IdealI \triangleleft \IdealJ \unlhd C^*(E)$ are ideals.
Then
\begin{enumerate}[(i)]
\item $C^*(E)\otimes\K$ has the corona factorization property 
\item $(\IdealJ/ \IdealI )\otimes\K$ is of the form $C^*(F)\otimes\K$ for some graph $F$
\item $(\IdealJ / \IdealI )\otimes\K$ has the corona factorization property
\end{enumerate}
\end{lemma}

The graph $F$ above can be chosen as a subgraph of the Drinen-Tomforde desingularization of $E$ (\cite{ddmt:cag}).

\begin{definition}
For each \Ca \A, we let $\multialg{\A}$ and $\corona{\A}$ denote the multiplier algebra and the corona algebra of \A, respectively. 

For each extension 
$$\mathfrak{e} \ : \ 0\to\B\to\mathfrak{E}\to\A\to 0,$$
we let $\ftn{\busby_\mathfrak{e}}{\A}{\corona{\B}}$ denote the Busby map of the extension. 

Moreover, for each surjective (or, more generally, proper) \shom $\ftn{\varphi}{\A}{\B}$, we let $\ftn{\widetilde{\varphi}}{\multialg{\A}}{\multialg{\B}}$ and $\ftn{\overline{\varphi}}{\corona{\A}}{\corona{\B}}$ denote the unique extension to the multiplier algebras and the induced \shom between the corona algebras, respectively (cf.~\S 2.1 of \cite{elp:mecpfbi}). 
\end{definition}

\begin{lemma} \label{l:dirsum}
Let $(\B_i)_{i\in I}$ be a family of \Cas (small enough for direct sums and products to exist). 
Let $\ftn{\pi_j}{\bigoplus_{i\in I}\B_i}{\B_j}$ denote the canonical projection, for each $j\in I$. 
Then there is a canonical isomorphism $\ftn{\prod_{i\in I}\widetilde{\pi_i}}{ \multialg{ \bigoplus_{ i \in I} \B_{i} } }{ \prod_{ i\in I} \multialg{ \B_{i} } }$ which has the unique extension $\ftn{\widetilde{\pi_j}}{\multialg{\bigoplus_{i\in I}\B_i}}{\multialg{\B_j}}$ of $\pi_j$ as the $j$'th coordinate map. 

Consequently, if $I$ is finite, there is an induced isomorphism $\ftn{ \prod_{i\in I} \overline{\pi_i} } { \corona{ \bigoplus_{ i\in I} \B_{i} } }{  \prod_{ i\in I} \corona{ \B_{i} } }$, and it induces homomorphisms $\ftn{\overline{\pi_j}}{\corona{ \bigoplus_{ i\in I} \B_{i} }}{\corona{\B_j}}$ as the $j$'th coordinate map. In this case, the direct product coincides with the direct sum. 
\end{lemma}

\begin{proof}
Here we view the multiplier algebras as the algebras of double centralisers (cf.~pp.~39 and~81-82 in \cite{gjm:cot}). 
Let $(\rho_1,\rho_2)$ be a double centralizer on $\bigoplus_{i\in I}\B_i$ (i.e., an arbitrary element of $\multialg{\bigoplus_{i\in I}\B_i}$). 
Using an approximate unit, it is easy to see that $\rho_1$ and $\rho_2$ restricted to $\B_j$ map into $\B_j$ itself. 
In this way we get a canonical \shom from $\multialg{\bigoplus_{i\in I}\B_i}$ to $\multialg{\B_j}$. 
By the universal property of the direct product, we get a \shom $\varphi$ from $\multialg{\bigoplus_{i\in I}\B_i}$ to $\prod_{i\in I}\multialg{\B_i}$, where the $j$'th coordinate map clearly is an extension of $\pi_j$ to the multiplier algebras, and hence it is \emph{the} extension $\widetilde{\pi_j}$ of $\pi_j$.  
Clearly, $\varphi$ is injective. It is also easy to show that $\varphi$ is surjective by constructing the preimage. 

Therefore, if $I$ is finite, the direct product of the short exact sequences 
$$\xymatrix{0\ar[r] & \B_j \ar[r] & \multialg{\B_j} \ar[r] & \corona{\B_j}\ar[r] & 0 }$$
is canonically isomorphic to 
$$\xymatrix{0\ar[r] & \bigoplus_{i\in I}\B_i \ar[r] & \multialg{\bigoplus_{i\in I}\B_i} \ar[r] & \corona{\bigoplus_{i\in I}\B_i}\ar[r] & 0 }$$
\end{proof}


\subsection{Primitive ideal space with $n$ maximal elements}\label{fanup}

\begin{assumption}
For this subsection, let $n>1$ be a fixed integer, and let $X_i=\X_{l_i}$ for $i=1,2,\ldots,n$, where $l_1,l_2,\ldots,l_n$ are fixed positive integers. 
Let, moreover,  
$$X=\{m\}\sqcup X_1\sqcup X_2 \sqcup \cdots \sqcup X_n$$
and define a partial order on $X$ as follows.
The element $m$ is the least element of $X$, and for each $i=1,2,\ldots,n$, if $x,y\in X_i$ then $x\leq y$ in $X$ if and only if $x\leq y$ in $X_i$. 
No other relations exist between the elements of $X$.
\end{assumption}

\begin{lemma}\label{l:essential}
Let \A be a tight \Ca over $X$ and let $k\in\{1,2,\ldots,n\}$ be given.  
Consider the extensions
\begin{equation*}
\mathfrak{e} \ : \ 0 \to \A ( X \setminus \{ m \} ) \to \A \to \A ( \{ m \} ) \to 0
\end{equation*}
and 
\begin{equation*}
\mathfrak{e} \cdot \pi_{k} \ : \ 0 \to \A ( X_k ) \to \A ( X_k\cup \{ m \} ) \to \A ( \{ m \} ) \to 0, 
\end{equation*}
where $\ftn{ \pi_{k} }{ \A( X \setminus \{ m \} ) }{ \A( X_k ) }$ is the canonical quotient \shom. 

Then $\busby_{ \mathfrak{e} \cdot \pi_{k} } = \overline{\pi}_{k} \circ \busby_{ \mathfrak{e} }$, and $\overline{\pi}_{k} \circ \busby_{ \mathfrak{e} }$ is injective. 
\end{lemma}

\begin{proof}
Note that the diagram 
\begin{equation*}
\xymatrix{
\mathfrak{e} \ : \ 0 \ar[r] & \A ( X \setminus \{ m \} ) \ar[r] \ar[d]^{ \pi_{k} } & \A \ar[r] \ar[d] & \A( \{ m \} ) \ar[r] \ar@{=}[d] & 0 \\
\mathfrak{e} \cdot \pi_{k} \ : \ 0 \ar[r] & \A ( X_k ) \ar[r] & \A( X_k\cup\{ m \} )  \ar[r] & \A( \{ m \} ) \ar[r] & 0
}
\end{equation*}
is commutative.  Since $\pi_{k}$ is surjective, by Theorem~2.2 of \cite{elp:mecpfbi}, $\overline{ \pi }_{k} \circ \busby_{ \mathfrak{e} } = \busby_{ \mathfrak{e} \cdot \pi_{k} }$. Also note, that Corollary~4.3 of \cite{elp:mecpfbi} justifies the notation $\mathfrak{e}\cdot\pi_k$. Suppose $\overline{ \pi }_{k} \circ \busby_{ \mathfrak{e} }$ is not injective, then $\overline{ \pi }_{k} \circ \busby_{ \mathfrak{e} } = 0$ since $\A ( \{ m \} )$ is a simple \Ca.  Hence, $\A ( X_k\cup\{ m \} ) \cong \A ( X_k ) \oplus \A ( \{ m \} )$.  
Since $\A(X_k\cup\{m\})\cong \A / \A(X\setminus(X_k\cup \{m\}))$, then there exist proper ideals \IdealI and \IdealJ of \A such that $\IdealI+\IdealJ=\A$ and $\IdealI\cap\IdealJ=\A(X\setminus(X_k\cup\{m\}))$. But this contradicts the fact that \A is a tight \Ca over $X$. Hence, $\overline{ \pi }_{k} \circ \busby_{ \mathfrak{e} }$ is injective. 
\end{proof}

\begin{lemma}\label{l:full}
Let \A be a tight \Ca over $X$.  Then 
\begin{equation*}
\mathfrak{e} \ : \ 0 \to \A ( X \setminus \{ m \} ) \to \A \to \A ( \{ m \} ) \to 0
\end{equation*}
is full if and only if $\mathfrak{e} \cdot \pi_{k}$ is full for all $k=1,2,\ldots,n$.  
\end{lemma}

\begin{proof}
By Lemma \ref{l:essential}, $\busby_{ \mathfrak{e}  \cdot \pi_{k} } = \overline{ \pi }_{k} \circ \busby_{ \mathfrak{e} }$.  Thus, if $\mathfrak{e}$ is a full extension, then $\mathfrak{e} \cdot \pi_{k}$ is a full extension since $\overline{ \pi }_{k}$ is surjective.  Suppose $\mathfrak{e} \cdot \pi_{k}$ is a full extension for all $k=1,2,\ldots,n$.  Note that $\A(X\setminus \{m\})$ is $\bigoplus_{j=1}^{n}\A(X_j)$ and thus from Lemma~\ref{l:dirsum} it follows that the $j$'th coordinate map of $\left(\bigoplus_{ i = 1}^{n} \overline{ \pi }_{i} \right) \circ \busby_{ \mathfrak{e} }$ is exactly $\overline{\pi_j}\circ\busby_{\mathfrak{e}}=\busby_{ \mathfrak{e} \cdot \pi_{j} }$ (according to Lemma~\ref{l:essential}).  Since $\bigoplus_{ i = 1}^{n} \overline{ \pi }_{i}$ is an isomorphism and since $\mathfrak{e} \cdot \pi_{k}$ is a full extension for all $k=1,2,\ldots,n$, we have that $\mathfrak{e}$ is a full extension. That this direct sum of full extensions is again full can easily be shown by first cutting down to each coordinate.  
\end{proof}

The signatures \myref{3}{6}{1}, \myref{3}{6}{5}, \myref{4}{39}{1}, \myref{4}{39}{3}, \myref{4}{39}{4}, \myref{4}{39}{5}, \myref{4}{39}{7}, \myref{4}{38}{1}, \myref{4}{38}{3}, \myref{4}{38}{7} are covered by the following theorem. 

\begin{theorem}\label{t:class1}
Let \A and \B be graph \Cas that are tight \Cas over $X$.  
Assume that there exists an isomorphism $\ftn{ \alpha }{ \FKplus{X}{\A} }{ \FKplus{X}{\B} }$. 
Assume, moreover, that $\A(\{m\} )$ is an \AFa and that $\Xaf$ is 
hereditary.  Then $\A \otimes \K \cong \B \otimes \K$. 
\end{theorem}

\begin{proof}
We may assume that \A and \B are stable \Cas. 
Note that for each $x\in X$, $\A ( \{ x \} )$ is an \AFa if and only if since $\B ( \{ x \} )$ is an \AFa, and $\A ( \{ x \} )$ is \Oia if and only if $\B ( \{ x \} )$ is \Oia (since there exists a positive isomorphism from $K_{0} ( \A ( \{x\} ))$ to $K_{0} ( \B ( \{x\}))$). Specifically, $\B(\{ m\})$ is an \AFa. 
First we assume that $\Xpi\neq\emptyset$ and $\Xaf\setminus\{m\}\neq\emptyset$. 

Note that $\A(\Xaf)$ and $\B(\Xaf)$ are \AFas. 
Since $\ftn{ \alpha_{ \Xaf }}{ K_{0} ( \A ( \Xaf) ) } { K_{0} ( \B ( \Xaf) ) }$ is a positive isomorphism, there exists an isomorphism $\ftn{ \beta }{ \A ( \Xaf) }{ \B ( \Xaf ) }$ such that $K_{0} ( \beta ) =  \alpha_{ \Xaf }$  (by Elliott's classification result \cite{gae:cilssfa}).  Since $\A ( \Xaf )$ and $\B ( \Xaf )$ are \AFas and $\beta$ is an $\Xaf$-equivariant isomorphism, we have that $K_{0} ( \beta_Y ) = \alpha_{ Y }$ for all $Y \in \mathbb{LC} ( X )$ such that $Y \subseteq \Xaf$.  In particular, $K_{0} ( \beta_{ \{ x \} } ) = \alpha_{ \{ x \} }$ for all $x \in \Xaf$.

Let $\Xpi^{\min}$ be the set of minimal elements of $\Xpi$, and for each $a,b\in X$ let
\begin{align*}
[a,\infty)&=\setof{x\in X}{a\leq x},\\
[a,b)&=\setof{x\in X}{a\leq x< b}.
\end{align*}

Let $x\in \Xpi^{\min}$ be given. Let $i_x\in\{1,2,\ldots,n\}$ be the unique number such that $x\in X_{i_x}$. 
Note that $X_{i_x}\sqcup\{m\} = [m,x)\cup[x,\infty)$, which we will denote by $\widetilde{X}_{i_x}$. Let, moreover, 
\begin{equation*}
\mathfrak{e}^\A_x \ : \ 0 \to \A ( [x,\infty) ) \to \A( \widetilde{X}_{i_x} ) \to \A ( [m,x) ) \to 0.
\end{equation*}
and
\begin{equation*}
\mathfrak{e}^\B_x \ : \ 0 \to \B ( [x,\infty) ) \to \B( \widetilde{X}_{i_x} ) \to \B ( [m,x) ) \to 0.
\end{equation*}
Since $\ftn{ \alpha }{ \FKplus{X}{\A} }{ \FKplus{X}{\B} }$ is an isomorphism, we also have an isomorphism $\ftn{ \alpha_{\widetilde{X}_{i_x}} }{ \FKplus{\widetilde{X}_{i_x}}{\A(\widetilde{X}_{i_x})} }{ \FKplus{\widetilde{X}_{i_x}}{\B(\widetilde{X}_{i_x})} }$. So by Theorem~4.14 of \cite{rmrn:ctsfk}, Kirchberg \cite{ek:nkmkna}, and Theorem~3.3 of \cite{segrer:ccfis}, there exists an isomorphism $\ftn{\varphi^x}{\A ([x,\infty))}{\B ([x,\infty))}$ 
such that $K_*(\varphi^x)=\alpha_{[x,\infty)}$, 
and
\begin{equation*}
\left[\busby_{\mathfrak{e}^\B_x}\circ\beta_{[m,x)}\right]=\left[\overline{\varphi^x}\circ\busby_{\mathfrak{e}^\A_x}\right]
\end{equation*}
in $\kk^1( \A( [m,x) ),\B( [x,\infty) ) )$, since $\kk(\beta_{[m,x)})$ is the unique lifting of $\alpha_{[m,x)}$. 

As in the proof of Proposition~6.3 of \cite{segrer:ccfis}, Corollary~5.3 of \cite{segrer:ccfis} implies that $\busby_{\mathfrak{e}^\A_x}$ and $\busby_{\mathfrak{e}^\B_x}$ are full extensions, and thus also the extensions with Busby maps $\busby_{\mathfrak{e}^\B_x}\circ\beta_{[m,x)}$ and $\overline{\varphi^x}\circ\busby_{\mathfrak{e}^\A_x}$ are full. 
Since the extensions are non-unital and $\B( [x,\infty) )$ satisfies the corona factorization property, there exists a unitary $u_x\in\multialg{\B( [x,\infty) ) }$ such that 
\begin{equation*}
\busby_{\mathfrak{e}^\B_x}\circ\beta_{[m,x)}=\Ad(\overline{u_x})\circ\overline{\varphi^x}\circ\busby_{\mathfrak{e}^\A_x} 
\end{equation*}
where $\overline{u_x}$ is the image of $u_x$ in the corona algebra (this follows from \cite{gaedk:avbfat} and \cite{dkpwn:cfpaue}). 
Hence, by Theorem 2.2 of \cite{elp:mecpfbi}, there exists an isomorphism $\ftn{\eta^x}{\A(\widetilde{X}_{i_x})}{\B(\widetilde{X}_{i_x})}$ such that $(\Ad(\overline{u_x})\circ\varphi^x,\eta^x,\beta_{[m,x)})$ is an isomorphism from $\mathfrak{e}^\A_x$ to $\mathfrak{e}^\B_x$. 
Let
\begin{equation*}
\mathfrak{e}^\A \ : \ 0 \to \A ( X \setminus \{ m \} ) \to \A \to \A ( \{ m \} ) \to 0,
\end{equation*}
and
\begin{equation*}
\mathfrak{e}^\B \ : \ 0 \to \B ( X \setminus \{ m \} ) \to \B \to \B ( \{ m \} ) \to 0.
\end{equation*}
Since $\A(\widetilde{X}_{i_x})$ and $\B(\widetilde{X}_{i_x})$ have linear ideal lattices, this induces an isomorphism 
$$\xymatrix{
\mathfrak{e}^\A\cdot\pi_{i_x}\colon & 0\ar[r] & \A(X_{i_x})\ar[d]^{\psi^x}\ar[r] & \A(\widetilde{X}_{i_x})\ar[d]\ar[r] & \A(\{m\}) \ar[d]^{\beta_{\{m\}}}\ar[r] & 0, \\ 
\mathfrak{e}^\B\cdot\pi_{i_x}\colon & 0\ar[r] & \B(X_{i_x})\ar[r] & \B(\widetilde{X}_{i_x})\ar[r] & \B(\{m\}) \ar[r] & 0.
}$$
So now by construction, 
$$\overline{\psi^x}\circ\busby_{\mathfrak{e}^\A\cdot\pi_{i_x}}=\busby_{\mathfrak{e}^\B\cdot\pi_{i_x}}\circ\beta_{\{m\}},$$
for all $x\in \Xpi^{\min}$, and
$$\overline{\beta_{X_j}}\circ\busby_{\mathfrak{e}^\A\cdot\pi_{j}}=\busby_{\mathfrak{e}^\B\cdot\pi_{j}}\circ\beta_{\{m\}},$$
for all $j=1,2,\ldots,n$ satisfying that $\A(X_j)$ is an \AFa. 
Now we define an isomorphism $\theta$ from $\A(X\setminus \{m\})$ to $\B(X\setminus \{m\})$ as the direct sum of the $\psi^x$'s and $\beta_{X_j}$'s. 
We get that (from Lemma~\ref{l:dirsum} and Lemma~\ref{l:essential}) 
\begin{gather*}
\overline{\theta}\circ\busby_{\mathfrak{e}^\A} =\overline{\theta}\circ
\left(\bigoplus_{j=1}^n\busby_{\mathfrak{e}^\A\cdot\pi_{j}}\right)
=\bigoplus_{j=1}^n\overline{\theta_j}\circ\busby_{\mathfrak{e}^\A\cdot\pi_{j}}
=\\
\bigoplus_{j=1}^n\busby_{\mathfrak{e}^\B\cdot\pi_{j}}\circ\beta_{\{m\}} =\left(\bigoplus_{j=1}^n\busby_{\mathfrak{e}^\B\cdot\pi_{j}}\right)\circ\beta_{\{m\}} =\busby_{\mathfrak{e}^\B}\circ\beta_{\{m\}} ,
\end{gather*}
where the $\theta_j$'s denote the corresponding $\psi^x$'s and $\beta_{X_j}$'s. 
Hence, by Theorem 2.2 of \cite{elp:mecpfbi}, $\A \cong \B$.

If $\Xpi=\emptyset$ the result is due to Elliott's classification result \cite{gae:cilssfa}, and if $\Xaf=\{m\}$ the theorem follows easily by making modifications to the above proof. 
\end{proof}

\begin{remark}\label{r:class1}
Let \A and \B be graph \Cas that are \Cas over $X$, so that $\A(X_i)$ and $\B(X_i)$ are tight \Cas over $X_i$, for $i=1,2,\ldots,n$. 
Assume that 
$$0\to \A(X_i)/\A(X_i\setminus\{x_i\})\to\A(X_i\cup\{m\})/\A(X_i\setminus\{x_i\})\to \A(X_i\cup\{m\})/\A(X_i)\to 0$$
is essential whenever $\A(X_i)$ is \Oia, where $x_i$ is the greatest element of $X_i$. 
Assume that there exists an isomorphism $\ftn{ \alpha }{ \FKplus{X}{\A} }{ \FKplus{X}{\B} }$. 
Assume moreover, that $\A ( \{ m \} )$ is an \AFa and that the set of $x\in X$ for which $\A(\{x\})$ is an \AFa is hereditary.  Then $\A \otimes \K \cong \B \otimes \K$. 
This follows from the proof above. 

The above extensions are essential, e.g., if $\A(\{x_i\})$ is the least ideal of $\A(\{x_i,m\})$, for all $i=1,2,\ldots,n$, and the remark applies to the cases\footnote{Here we specify how we view the algebras as algebras over $a\leftarrow b\rightarrow c$ by providing a continuous map from the primitive ideal space to $\{a,b,c\}$} 
\begin{enumerate}[(a)]
\item
\myref{4}{E}{1}, where we view the algebra \A that is tight over the space \msp{4}{E} as a \Ca over $a\leftarrow b \rightarrow c$ as indicated by the assignment $b\rightarrow a\leftarrow b\rightarrow c$.

\item
\myref{4}{1E}{1} and 
\myref{4}{1E}{3}, where we view the algebra \A that is tight over the space \msp{4}{1E} as a \Ca over $a\leftarrow b \rightarrow c$ as indicated by the assignment 
\item
\myref{4}{3E}{1}, where we view the algebra \A that is tight over the space \msp{4}{3E} as a \Ca over $a\leftarrow b \rightarrow c$ as indicated by the assignment
\end{enumerate}
\end{remark}

The following proposition follows from the results in \cite{semt:cnga}.

\begin{proposition}\label{p:full}
Let \A be a graph \Ca with exactly one nontrivial ideal \IdealI.  If \A is not an \AFa, then $0 \to \IdealI \otimes \K \to \A\otimes \K \to \A / \IdealI \otimes \K \to 0$ is a full extension.
\end{proposition}

Using the UCT for accordion spaces (see \cite{rmrn:ctsfk} and \cite{rbmk:uctcfts}) and for many other four-point spaces under the added assumption of real rank zero as described in \cite{sagrer:fkrrzc}, the cases \myref{3}{6}{2}, \myref{3}{6}{3}, \myref{4}{38}{8}, \myref{4}{38}{9}, \myref{4}{38}{11}, can be classified using the following theorem. 

\begin{theorem}\label{t:class1-PI}
Let $\A$ and $\B$ be graph \Cas that are tight \Cas over $X$, with $X_i$ being a singleton, for each $i=1,2,\ldots,n$.
Suppose there exists an isomorphism $\ftn{ \alpha }{ \FKplus{X}{\A} }{ \FKplus{X}{\B} }$ which lifts to an invertible element in $\kk ( X ; \A , \B )$.  Then $\A \otimes \K \cong \B \otimes \K$.  
\end{theorem}

\begin{proof}
If $\A ( \{ m\} )$ is an \AFa, the result follows from Theorem \ref{t:class1}.  Suppose $\A ( \{ m \} )$ is an \Oia simple \Ca and that \A and \B are stable \Cas.  Then by Lemma~\ref{l:essential} and Proposition~\ref{p:full}, $\ftn{ \overline{\pi}_{ i } \circ \busby_{ \mathfrak{e}^\A } }{ \A ( \{ m \} ) }{ \corona{ \A (X_i)} }$  and $\ftn{ \overline{\pi}_{ i } \circ \busby_{ \mathfrak{e}^\B } }{ \B ( \{ m \} ) }{ \corona{ \B (X_i) } }$  are full extensions, for all $i=1,2,\ldots,n$.  Hence, by Lemma \ref{l:full}, $\busby_{ \mathfrak{e}^\A }$ and $\busby_{ \mathfrak{e}^\B }$ are full extensions.  The theorem now follows from the results of \cite{segrer:ccfis}.
\end{proof}

\subsection{Primitive ideal space with $n$ minimal elements}\label{fandown}

\begin{assumption}
For this subsection, let $n>1$ be a fixed integer, and let $X_i=\X_{l_i}$ for $i=1,2,\ldots,n$, where $l_1,l_2,\ldots,l_n$ are fixed positive integers. 
Let, moreover,  
$$X=\{M\}\sqcup X_1\sqcup X_2 \sqcup \cdots \sqcup X_n$$
and define a partial order on $X$ as follows.
The element $M$ is the greatest element of $X$, and for each $i=1,2,\ldots,n$, if $x,y\in X_i$ then $x\leq y$ in $X$ if and only if $x\leq y$ in $X_i$. 
No other relations are between the elements of $X$.
\end{assumption}


\begin{lemma}\label{l:pullback}
Let \A be a tight \Ca over $X$ and let $Y\in\mathbb{O}(\Xaf\setminus\{M\})$ be given. 
Consider the extensions
\begin{equation*}
\mathfrak{e} \ : \ 0 \to \A ( \{ M \} ) \to \A \to \A( X \setminus \{ M \} ) \to 0
\end{equation*}
and 
\begin{equation*}
\iota_{\A,Y} \cdot \mathfrak{e} \ : \ 0 \to \A ( \{ M \} ) \to \A ( Y\cup\{  M \} ) \to \A ( Y ) \to 0 
\end{equation*}
where $\ftn{ \iota_{\A,Y} }{ \A ( Y ) }{ \A ( X \setminus \{ M \} ) }$ is the usual embedding.  Then $\busby_{ \iota_{\A,Y} \cdot \mathfrak{e} } = \busby_{ \mathfrak{e} } \circ \iota_{\A,Y}$.
\end{lemma}

\begin{proof}
Note that the diagram 
\begin{equation*}
\xymatrix{
0 \ar[r] & \A ( \{ M \} ) \ar[r] \ar@{=}[d] & \A ( Y\cup\{ M \} ) \ar[r] \ar[d] & \A ( Y ) \ar[r] \ar[d]^{ \iota_{\A,Y} } &  0 \\
0 \ar[r] & \A ( \{ M \} ) \ar[r] & \A  \ar[r] & \A ( X \setminus \{ M \} ) \ar[r] &  0}
\end{equation*}
commutes. 
Hence, by Theorem 2.2 of \cite{elp:mecpfbi}, $\busby_{ \iota_{\A,Y} \cdot \mathfrak{e} } = \busby_{ \mathfrak{e} } \circ \iota_{ \A,Y }$.
\end{proof}

\begin{lemma}\label{l:heralg}
Suppose the following diagram of \Cas with short exact rows is commutative
\begin{equation*}
\xymatrix{
0 \ar[r] & \B \ar[r]^{ \iota_1 } \ar@{=}[d] & \mathfrak{E}_{1} \ar[r]^{ \pi_1 } \ar[d]^{ \varphi_{1} } & \A_{1} \ar[r] \ar[d]^{ \varphi_{2} }  & 0 \\
0 \ar[r] & \B \ar[r]^{ \iota_2 } & \mathfrak{E}_{2} \ar[r]^{ \pi_2 } & \A_{2} \ar[r] & 0.
}
\end{equation*}
\begin{enumerate}[(1)]
\item  If $\varphi_{2} ( \A_{1} )$ is a hereditary sub-\Ca of $\A_{2}$, then $\varphi_{1} ( \mathfrak{E}_{1} )$ is a hereditary sub-\Ca of $\mathfrak{E}_{2}$.

\item  If $\varphi_{2} ( \A_{1} )$ is full in $\A_{2}$, then $\varphi_{1} ( \mathfrak{E}_{1} )$ is full in $\mathfrak{E}_{2}$.
\end{enumerate}
\end{lemma}

\begin{proof}
We first prove (1).  Let $x \in \mathfrak{E}_{1}$ and $y \in \mathfrak{E}_{2}$ such that $0 \leq y \leq  \varphi_{1} (x)$.  Since $\varphi_{2} ( \A_{1} )$ is a hereditary sub-\Ca of $\A_{2}$, we have that there exists $z \in \varphi_{1} ( \mathfrak{E}_{1} )$ such that $\pi_{2} (y) = \pi_2(z)$.  Thus, $y - z \in \B$.  Since the map on the ideals is the identity, we have that $y - z \in \varphi_{1} ( \mathfrak{E}_{1} )$.  Hence, $y \in \varphi_{1} ( \mathfrak{E}_{1} )$.  Therefore, $\varphi_{1} ( \mathfrak{E}_{1} )$ is a hereditary sub-\Ca of $\mathfrak{E}_{2}$.

We now prove (2).  Let $x \in \mathfrak{E}_{2}$.  Since $\varphi_{2} ( \A_{1} )$ is full in $\A_{2}$, there exists $y$ in the ideal of $\mathfrak{E}_{2}$ generated by $\varphi_{1} ( \mathfrak{E}_{1} )$ such that  $x - y \in \B$.  Since the map on the ideals is the identity, we have that $y - z \in \varphi_{1} ( \mathfrak{E}_{1} )$. Hence, $x$ is in the ideal of $\mathfrak{E}_{2}$ generated by $\varphi_{1} ( \mathfrak{E}_{1} )$.  
\end{proof}

\begin{lemma}\label{l:directsumfull}
Let $\mathfrak{e} : 0 \to \IdealI \to \A \to \bigoplus_{ k  =1}^{n} \A_{k} \to 0$ be an extension and let $\iota_{k} : \A_{k} \to \bigoplus_{ k  =1}^{n} \A_{k}$ be the inclusion.  Suppose $\busby_{ \mathfrak{e} } \circ \iota_{k}$ is full for each $k$.  Then $\busby_{ \mathfrak{e} }$ is full.  
\end{lemma}

\begin{proof}
Let $(a_{1} , a_{2} , \dots, a_{n} )$ be a nonzero positive element in $\bigoplus_{ k  =1}^{n} \A_{k}$.  Without loss of generality, we may assume that $a_{1} \neq 0$.  Note that ideal in $\corona{ \IdealI }$ generated by $\busby_{ \mathfrak{e} } ( a_{1} , \dots, a_{n} )$ contains the ideal in $\corona{\IdealI }$ generated by $\busby_{ \mathfrak{e} } \circ \iota_{1}(a_{1})$.  Since $\busby_{ \mathfrak{e} } \circ \iota_{k}$ is full, we have that the ideal in $\corona{ \IdealI }$ generated by $\busby_{ \mathfrak{e} } \circ \iota_{1} (a_{1})$ is $\corona{ \IdealI}$.  Thus, the ideal in $\corona{ \IdealI }$ generated by $\busby_{ \mathfrak{e} } ( a_{1} , \dots, a_{n} )$ is $\corona{\IdealI}$.
\end{proof}

The following result applies to the cases \myref{3}{3}{1}, \myref{3}{3}{5}, \myref{4}{F}{6}, \myref{4}{F}{8}, \myref{4}{F}{14}, \myref{4}{A}{2}, \myref{4}{A}{6}, \myref{4}{A}{14}.

\begin{theorem}\label{t:class2}
Let $\A$ and $\B$ be graph \Cas that are tight \Cas over $X$ such that each of $\A(X_{i})$, $\B(X_{i})$ are either \AFas or \Oia.  Suppose there exists an isomorphism $\ftn{ \alpha }{ \FKplus{X}{\A} }{ \FKplus{X}{\B} }$ and $\A ( \{ M \} )$ is an \AFa.  Then $\A \otimes \K \cong \B \otimes \K$.  
\end{theorem}

\begin{proof}
We may assume that \A and \B are stable \Cas.  Note that for each $x\in X$, $\A ( \{ x \} )$ is an \AFa if and only if $\B ( \{ x \} )$ is an \AFa, and $\A ( \{ x \} )$ is \Oia if and only if $\B ( \{ x \} )$ is \Oia (since there exists a positive isomorphism from $K_{0} ( \A ( \{x\} )$ to $K_{0} ( \B ( \{x\})$). Specifically, $\B(\{ M\})$ is an \AFa. 
First we assume that $\Xpi\neq\emptyset$ and $\Xaf\setminus\{M\}\neq\emptyset$. 

Note that $\A (\Xaf )$ and $\B ( \Xaf )$ are \AFas.  
Since $\ftn{ \alpha_{ \Xaf } }{ K_{0} ( \A ( \Xaf ) ) } { K_{0} ( \B ( \Xaf ) ) }$ is a positive isomorphism, there exists an isomorphism $\ftn{ \beta }{ \A ( \Xaf ) }{ \B ( \Xaf ) }$ such that $K_{0} ( \beta ) =  \alpha_{ \Xaf }$ (by Elliott's classification result \cite{gae:cilssfa}). 
Since $\A ( \Xaf )$ and $\B ( \Xaf )$ are \AFas and $\beta$ is an $\Xaf$-equivariant isomorphism, we have that $K_{0} ( \beta_Y ) = \alpha_{ Y }$ for all $Y \in \mathbb{LC} ( X )$ such that $Y \subseteq \Xaf$.  In particular, $K_{0} ( \beta_{ \{ x \} } ) = \alpha_{ \{ x \} }$ for all $x \in \Xaf$. 

Let 
\begin{equation*}
\mathfrak{e}_\A \ : \ 0 \to \A ( \{M \} ) \to \A \to \A ( X \setminus \{M\} ) \to 0,
\end{equation*}
and
\begin{equation*}
\mathfrak{e}_\B \ : \ 0 \to \B ( \{M \} ) \to \B \to \B ( X \setminus \{M\} ) \to 0.
\end{equation*}
Since $\beta$ is an $\Xaf$-equivariant isomorphism, by Lemma~\ref{l:pullback} above and Theorem~2.2 of \cite{elp:mecpfbi}, for $Y \in \mathbb{O} ( \Xaf\setminus\{M\})$
\begin{equation*}
\overline{ \beta }_{ \{ M \} }  \circ \busby_{ \mathfrak{e}_{\A} } \circ \iota_{ \A, Y } = \busby_{ \mathfrak{e}_{B} } \circ \iota_{\B, Y } \circ \beta_{ Y }
\end{equation*}
for all $Y  \in \mathbb{O} ( \Xaf\setminus\{M\})$, where $\ftn{ \iota_{\A, Y} }{ \A ( Y ) }{ \A ( X \setminus \{M\} ) }$ and $\ftn{ \iota_{\B, Y} }{ \B ( Y ) }{ \B ( X \setminus \{M\} ) }$ are the canonical embeddings.

Since $\alpha$ induces an isomorphism reaching  from $\FKplus{ \Xpi\cup \{ M \} }{\A ( \Xpi\cup \{ M \}  ) }$ to $\FKplus{ \Xpi \cup \{ M \} }{\B ( \Xpi \cup \{ M \}  ) }$, by Lemma~\ref{l:pullback}, Theorem~2.3 of \cite{segrer:cecc}, Theorem~4.14 of \cite{rmrn:ctsfk}, Kirchberg \cite{ek:nkmkna}, and Theorem~3.3 of \cite{segrer:ccfis}), there exists an $\Xpi $-equivariant isomorphism $\ftn{ \psi }{ \A ( \Xpi) }{ \B ( \Xpi ) }$ such that $K_{*} ( \psi ) = \alpha_{ \Xpi}$ and 
\begin{align*}
[ \overline{\beta}_{ \{ M \} } \circ \busby_{ \mathfrak{e}_{\A} } \circ \iota_{ \A , \Xpi } ] = [ \busby_{ \mathfrak{e}_{\B} } \circ \iota_{ \B , \Xpi } \circ \psi ]
\end{align*}
in $\kk^{1} ( \A ( \Xpi) , \B ( \{M \} ) )$.  By Corollary~5.6 of \cite{segrer:ccfis}, $\busby_{ \mathfrak{e}_{\A} } \circ \iota_{ \A , X_{i} }$ and $\busby_{ \mathfrak{e}_{\B} } \circ \iota_{ \B , X_{i} }$ are full extensions for each $i=1,2,\ldots,n$ with $X_i$ being \Oia (i.e., $X_i\subseteq\Xpi$). Thus, $\busby_{ \mathfrak{e}_{\A} } \circ \iota_{ \A , \Xpi }$ and $\busby_{ \mathfrak{e}_{\B} } \circ \iota_{ \B , \Xpi }$ are full extensions since $\A ( \Xpi ) = \bigoplus_{i\in\{1,2,\ldots,n\}, X_i\subseteq\Xpi} \A(  X_{ i } )$ and $\B ( \Xpi ) = \bigoplus_{i\in\{1,2,\ldots,n\}, X_i\subseteq\Xpi } \B(  X_{ i } )$.  Hence, $\overline{\beta}_{ \{ M \} } \circ \busby_{ \mathfrak{e}_{\A} } \circ \iota_{ \A , \Xpi }$ and $\busby_{ \mathfrak{e}_{\B} } \circ \iota_{ \B , \Xpi } \circ \psi$ are full extensions.

Let $\ftn{ \pi_{\A, \Xaf\setminus\{M\}} }{ \A ( X \setminus \{ M \} ) }{ \A ( \Xaf\setminus\{M\} ) }$, $\ftn{ \pi_{\A, \Xpi } }{ \A ( X \setminus \{ M \} ) }{ \A ( \Xpi ) }$, $\ftn{ \pi_{\B, \Xaf\setminus\{M\} } }{ \B ( X \setminus \{ M \} ) }{ \B ( \Xaf\setminus\{M\} ) }$, $\ftn{ \pi_{\B, \Xpi } }{ \B ( X \setminus \{ M \} ) }{ \B ( \Xpi ) }$ be the canonical projections.  Note that the range of $\busby_{ \mathfrak{e}_{\A} } \circ \iota_{ \A,  \Xaf\setminus\{M\} }$ and the range of $\busby_{ \mathfrak{e}_{\A} } \circ \iota_{ \A, \Xpi }$ are orthogonal and the range of $\busby_{ \mathfrak{e}_{\B} } \circ \iota_{ \B,  \Xaf\setminus\{M\} }$ and the range of $\busby_{ \mathfrak{e}_{\B} } \circ \iota_{ \B, \Xpi }$ are orthogonal.  Moreover, 
\begin{align*}
\busby_{ \mathfrak{e}_{\A} } &= \busby_{ \mathfrak{e}_{\A } } \circ \iota_{ \A, \Xaf\setminus\{M\} } \circ \pi_{\A,\Xaf\setminus\{M\}} + \busby_{ \mathfrak{e}_{\A} } \circ \iota_{ \A , \Xpi } \circ \pi_{ \A, \Xpi } \\
\busby_{ \mathfrak{e}_{\B} } &= \busby_{ \mathfrak{e}_{\B } } \circ \iota_{ \B, \Xaf\setminus\{M\} } \circ \pi_{\B,\Xaf\setminus\{M\}} + \busby_{ \mathfrak{e}_{\B} } \circ \iota_{ \B , \Xpi } \circ \pi_{ \B, \Xpi }.
\end{align*}

We claim that there exist full hereditary sub-\Cas $\mathcal{E}_{1}$ and $\mathcal{E}_{2}$ of $\A$ and $\B$, respectively, such that $\mathcal{E}_{1} \cong \mathcal{E}_{2}$.  Then by Theorem~2.8 of \cite{lgb:sihsc}, $\A \otimes \K \cong \B \otimes \K$.

Choose full projections $p_{1}, q_{1} \in \A ( \Xpi )$ and $p_{2}, q_{2} \in \A ( \Xaf\setminus\{M\} )$ such that $p_{1} + p_{2}$ is orthogonal to $q_{1} + q_{2}$ in $\A(X\setminus\{M\})$ (to do this, we use stability, and that graph algebras with finitely many ideals satisfies Condition (K) and hence are of real rank zero).  Therefore, $\busby_{ \mathfrak{e}_{\A} } ( p_{1} + p_{2} ) \neq 1_{ \corona{ \A ( \{ M \} ) } }$ since $\busby_{ \mathfrak{e}_{\A} } ( p_{1} + p_{2} )$ is orthogonal to $\busby_{ \mathfrak{e}_{\A} } ( q_{1} + q_{2} )$.  Set $e_{1} = \psi ( p_{1} )$, $e_{2} = \beta_{ \Xaf\setminus\{M\} } ( p_{2} )$, $f_{1} = \psi ( q_{1} )$, and $f_{2} = \beta_{ \Xaf\setminus\{M\} } ( q_{2} )$.  Then $e_{1} + e_{2}$ and $f_{1} + f_{2}$ are nonzero orthogonal projections.  So, $\busby_{ \mathfrak{e}_{\B} } ( e_{1} + e_{2} ) \neq 1_{ \corona{ \B ( \{ M \} ) } }$.

Set $e = \overline{ \beta }_{ \{ M \} } \circ \busby_{ \mathfrak{e}_{ \A } } \circ \iota_{ \A, \Xaf\setminus\{M\} } ( p_{2} ) = \busby_{ \mathfrak{e}_{ \B } } \circ \iota_{ \B, \Xaf\setminus\{M\} } \circ \beta_{ \Xaf\setminus\{M\} } ( p_{2} )$ and set $f = (1_{ \corona{ \B ( \{ M \} )  } } - e )$.  Let $\ftn{ j_{\pisy } }{ p_{1}  \A ( \Xpi ) p_{1} }{ \A ( \Xpi ) }$ and $\ftn{ j_{\afsy } }{ p_{2} \A ( \Xaf\setminus\{M\} ) p_{2} }{ \A ( \Xaf\setminus\{M\} ) }$ be the usual embeddings.  Note that 
\begin{equation*}
e \overline{ \beta }_{ \{ M \} } \circ \busby_{ \mathfrak{e}_{ \A } } \circ \iota_{ \A , \pisy} \circ j_{\pisy} ( x ) = \overline{ \beta }_{ \{ M \} } \circ \busby_{ \mathfrak{e}_{ \A } }  \circ \iota_{\A , \pisy  } \circ j_{\pisy } ( x )e =  0
\end{equation*} 
and 
\begin{gather*}
e \left( \busby_{ \mathfrak{ e }_{\B} } \circ \iota_{\B, \Xpi }  \circ
  \psi \circ j_{ \pisy} ( x )\right)=\\ \left( \busby_{ \mathfrak{e}_{
      \B } } \circ \iota_{ \B, \Xaf\setminus\{M\} }  \circ \beta_{
    \Xaf\setminus\{M\} } ( p_{2} ) \right) \cdot \left( \busby_{
    \mathfrak{ e }_{\B} } \circ \iota_{\B, \Xpi }  \circ \psi \circ
  j_{ \pisy} ( x ) \right)= 0 
\end{gather*}
as well as
\begin{gather*}
 \left( \busby_{ \mathfrak{ e }_{\B} } \circ \iota_{\B, \Xpi }  \circ \psi \circ j_{ \pisy} ( x )\right)e =\\  \left( \busby_{ \mathfrak{ e }_{\B} } \circ \iota_{\B, \Xpi }  \circ \psi \circ j_{ \pisy } ( x ) \right) \cdot \left( \busby_{ \mathfrak{e}_{ \B } } \circ \iota_{ \B, \Xaf\setminus\{M\} } \circ \beta_{ \Xaf\setminus\{M\} } ( p_{2} )  \right)= 0
\end{gather*}
for all $x \in p_{1}  \A ( \Xpi ) p_{1}$.  Hence, we have injective homomorphisms $ \overline{ \beta }_{ \{ M \} } \circ \busby_{ \mathfrak{e}_{ \A } } \circ \iota_{ \A , \pisy } \circ j_{\pisy} $ and $\busby_{ \mathfrak{ e }_{\B} } \circ \iota_{\B, \Xpi }  \circ \psi \circ j_{ \pisy }$ from $p_{1} \A (  \Xpi  )p_{1}$ to $f \corona{ \B ( \{M\} )  } f$.

Since $\B ( \{ M \} )$ is an \AFa, by Corollary 2.11 of \cite{sz:kqimp} $f$ lifts to a projection $f'$ in $\multialg{ \B( \{ M \} ) }$.  Note that there exists an isomorphism $\gamma$ from $f' \multialg{ \B ( \{M\} ) } f'$ to $\multialg{ f' \B ( \{ M \} )f' }$ which is the identity on $f' \B ( \{ M \} ) f'$ (see II.7.3.14, pp.~147 of \cite{bb:book}).  Thus, we have an isomorphism $\overline{\gamma}$ from $f \corona{ \B ( \{M\} ) } f$ to $\corona{ f' \B ( \{ M \} )f' }$ such that the diagram 
\begin{align*}
\xymatrix{
0 \ar[r] & f' \B ( \{ M \} ) f' \ar[r] \ar@{=}[d] & f' \multialg{ \B ( \{ M \} ) } f' \ar[r] \ar[d]^{ \gamma } & f \corona{ \B ( \{ M \} ) } f \ar[r] \ar[d]^{ \overline{\gamma} } & 0 \\
0 \ar[r] & f' \B ( \{ M \} ) f' \ar[r] & \multialg{ f' \B ( \{ M \} ) f' } \ar[r] & \corona{ f' \B ( \{ M \} ) f' } \ar[r] & 0 
}
\end{align*}
is commutative.  By Corollary~5.6 of \cite{segrer:ccfis}, $\busby_{ \mathfrak{e}_{\A} } \circ \iota_{ \A , X_{i} }$ and $\busby_{ \mathfrak{e}_{\B} } \circ \iota_{ \B , X_{i} }$ are full extensions for each $i=1,2,\ldots,n$ with $X_i$ being \Oia (i.e., $X_i\subseteq\Xpi$).  Thus, by Lemma~\ref{l:directsumfull}, $\busby_{ \mathfrak{e}_{\A} } \circ \iota_{ \A , \Xpi }$ and $\busby_{ \mathfrak{e}_{\B} } \circ \iota_{ \B , \Xpi }$ are full extensions since $\A ( \Xpi ) = \bigoplus_{i\in\{1,2,\ldots,n\}, X_i\subseteq\Xpi} \A(  X_{i} )$ and $\B ( \Xpi ) = \bigoplus_{i\in\{1,2,\ldots,n\}, X_i\subseteq\Xpi} \B(  X_{i} )$.  Hence, $\overline{\beta}_{ \{ M \} } \circ \busby_{ \mathfrak{e}_{\A} } \circ \iota_{ \A , \Xpi }$ and $\busby_{ \mathfrak{e}_{\B} } \circ \iota_{ \B , \Xpi } \circ \psi$ are full extensions.  Thus, $\overline{ \beta }_{ \{ M \} }  \circ \busby_{ \mathfrak{e}_{\A} } \circ \iota_{\A, \Xpi } ( p_{1} )$ is a norm-full projection in $\corona{ \B ( \{M\} ) }$.  Since $\overline{ \beta }_{ \{ M \} }  \circ \busby_{ \mathfrak{e}_{\A} } \circ \iota_{\A, \Xpi } ( p_{1} ) \leq f$, we have that $f$ is a norm-full projection in $\corona{ \B ( \{ M \} ) }$.  By Lemma~3.3 of \cite{segrer:okfe}, we have that $f'$ is a norm-full projection in $\multialg{ \B ( \{ M \} ) }$ since $\B ( \{ M \} )$ has an approximate identity consisting of projections.  Since $\B ( \{ M \} )$ is an \AFa, by Lemma~3.10 of \cite{segrer:cecc}, $\B (\{ M \})$ has the corona factorization property.  Thus, $f'$ is Murray-von Neumann equivalent to $1_{\multialg{ \B ( \{ M \} )  }}$.  Thus, $f' \B ( \{M\} )  f' \cong \B ( \{ M \} )$ which implies that $f' \B ( \{M\} )  f'$ is a stable \Ca since $\B ( \{ M \} )$ is a stable \Ca.
 
Let $\iota$ be the embedding of $f' \B ( \{M\} )  f'$ into $\B(\{M\})$, $\widetilde{\iota}$ be the embedding of $f' \multialg{ \B ( \{M\} ) }f'$ into $\multialg{\B( \{ M \} ) }$, and $\overline{\iota}$ be the embedding of $f \corona{ \B ( \{M\} ) } f$ into $\corona{ \B ( \{M\} ) }$.  Note that the following diagram 
\begin{align*}
\xymatrix{
0 \ar[r] & f' \B ( \{M\} ) f' \ar[r] \ar[d]^{ \iota } & f' \multialg{ \B ( \{ M \} ) } f' \ar[r] \ar[d]^{ \widetilde{\iota} } & f \corona{ \B ( \{ M \} ) } f \ar[r] \ar[d]^{ \overline{\iota} } & 0 \\
0 \ar[r] & \B ( \{ M \} ) \ar[r] &  \multialg{ \B ( \{ M \} ) }  \ar[r] & \corona{ \B ( \{ M \} ) }  \ar[r] & 0 
}
\end{align*} 
is commutative.  Note that the range of $\busby_{ \mathfrak{e}_{\B } } \circ \iota_{\B, \Xpi } \circ \psi \circ j_{\pisy}$ and the range of $\overline{ \beta}_{ \{ M \} } \circ \busby_{ \mathfrak{e}_{ \A }  } \circ \iota_{\A, \Xpi } \circ j_{\pisy }$ are contained in $f \corona{ \B ( \{ M \} ) } f $.  Let $\mathfrak{e}_{1}$ be the extension defined by $\overline{ \gamma } \circ \overline{ \iota}^{-1} \circ \overline{ \beta}_{ \{ M \} } \circ \busby_{ \mathfrak{e}_{ \A }  } \circ \iota_{\A, \Xpi } \circ j_{\pisy }$ and let $\mathfrak{e}_{2}$ be the extension defined by $\overline{\gamma} \circ \overline{\iota}^{-1} \circ \busby_{ \mathfrak{e}_{\B } } \circ \iota_{\B, \Xpi } \circ \psi \circ j_{\pisy}$.  Then 
\begin{align*}
\overline{ \iota } \circ \overline{ \gamma }^{-1} \circ \busby_{ \mathfrak{e}_{1} } = \overline{ \beta}_{ \{ M \} } \circ \busby_{ \mathfrak{e}_{ \A }  } \circ \iota_{\A, \Xpi } \circ j_{\pisy }
\quad \text{and} \quad
\overline{ \iota } \circ \overline{ \gamma }^{-1} \circ \busby_{ \mathfrak{e}_{2} } = \busby_{ \mathfrak{e}_{\B } } \circ \iota_{\B, \Xpi } \circ \psi \circ j_{\pisy}
\end{align*}

Since $\busby_{ \mathfrak{e}_{\A} } ( p_{1} + p_{2} ) \neq 1_{ \corona{ \A ( \{ M \} )  } }$ and $\busby_{ \mathfrak{e}_{\B} } ( e_{1} + e_{2} ) \neq 1_{ \corona{ \B ( \{ M \} )  } }$ and since $\overline{ \beta }_{ \{ M \} }$ and $\psi$ are isomorphisms, we have that $\overline{ \beta}_{ \{ M \} } \circ \busby_{ \mathfrak{e}_{ \A } } \circ \iota_{\A, \Xpi} \circ j_{ \pisy } ( p_{1} ) \neq f$ and $\busby_{ \mathfrak{ e }_{ \B  } } \circ \iota_{ \B , \Xpi } \circ \psi \circ  j_{ \pisy} ( p_{1}  ) \neq f$.  Thus, $\busby_{ \mathfrak{e}_{1} } ( p_{1} )$ and $\busby_{ \mathfrak{e}_{2} } (p_{1})$ are not equal to $1_{ \corona{ f' \B ( \{ M \} ) f' }  }$.  Therefore, $\mathfrak{e}_{1}$ and $\mathfrak{e}_{2}$ are non-unital full extensions.  Since 
\begin{align*}
[ \overline{\beta}_{ \{ M \} } \circ \busby_{ \mathfrak{e}_{\A} } \circ \iota_{ \A , \Xpi } ] = [ \busby_{ \mathfrak{e}_{\B} } \circ \iota_{ \B , \Xpi } \circ \psi ]
\end{align*}
in $\kk^{1} ( \A ( \Xpi) , \B ( \{M \} ) )$, since $\iota$ induces an
element in $\kk( f' \B ( \{ M \} ) f' , \B ( \{ M \} ) )$ which is invertible, and since $\overline{\gamma}$ is an isomorphism, we have that $[\busby_{ \mathfrak{e}_{1} } ] = [ \busby_{ \mathfrak{e}_{2} }  ]$ in $\kk^{1} ( p_{1}  \A ( \Xpi ) p_{1} , f'  \B ( \{ M \} ) f' )$.  Since $f'  \B ( \{ M \} )f' \cong \B ( \{ M \} )$, we have that $f'  \B ( \{ M \} ) f'$ has the corona factorization property.  Thus, there exists a unitary $u'$ in $\multialg{ f' \B ( \{ M \} ) f'} $ such that 
\begin{align*}
\mathrm{Ad} ( \overline{u'} ) \circ \busby_{ \mathfrak{e}_{1} } = \busby_{ \mathfrak{e}_{2} }, 
\end{align*}
where $\overline{u'}$ is the image of $u'$ in $\corona{ f' \B ( \{ M \} ) f' }$.  Let $u = \widetilde{\iota} \circ \gamma^{-1} (u')$.  Then $u$ is a partial isometry in $\multialg{ \B ( \{ M \} ) }$ such that $u^{*}u = f' = u u^{*}$ and 
\begin{equation*}
\mathrm{Ad} ( \overline{u} ) \circ \overline{ \beta}_{ \{ M \} } \circ \busby_{ \mathfrak{e}_{ \A }  } \circ \iota_{\A, \Xpi } \circ j_{\pisy}  =  \busby_{ \mathfrak{e}_{\B } } \circ \iota_{\B, \Xpi } \circ \psi \circ j_{\pisy} 
\end{equation*}
where $\overline{u}$ is the image of $u$ in $\corona{ \B ( \{ M \} ) }$.  Set $v = u + 1_{ \multialg{ \B ( \{ M \} ) } } - f'$ and let $\overline{v}$ be the image of $v$ in $\corona{ \B ( \{ M \} ) }$.  Note that $\overline{v} = \overline{u} + e$ and
\begin{align*}
\mathrm{Ad} ( \overline{v} ) \circ \overline{ \beta }_{ \{ M \} } \circ \busby_{ \mathfrak{e}_{ \A } } \circ \iota_{ \A, \Xaf\setminus\{M\} } \circ j_{\afsy} =  \overline{ \beta}_{ \{ M \} } \circ \busby_{ \mathfrak{e}_{  \A } } \circ \iota_{\A, \Xaf\setminus\{M\} } \circ j_{ \afsy} \\
\mathrm{Ad} ( \overline{v} ) \circ \overline{ \beta }_{ \{ M \} } \circ \busby_{ \mathfrak{e}_{ \A } } \circ \iota_{\A, \Xpi } \circ j_{ \pisy }  =  \busby_{ \mathfrak{e}_{\B } } \circ \iota_{\B, \Xpi } \circ \psi \circ j_{ \pisy }.  
 \end{align*}

Let $a_{1} \in p_{1} \A ( \Xpi) p_{1}$ and $a_{2} \in p_{2} \A ( \Xaf\setminus\{M\} ) p_{2}$.  Then   
\begin{align*}
&\overline{v} \left(  \overline{ \beta }_{ \{ M \} } \circ \busby_{ \mathfrak{e}_{\A} } \circ \iota_{\A, \Xpi } \circ j_{\pisy }( a_{1} ) + \overline{ \beta }_{ \{ M \} } \circ \busby_{ \mathfrak{e}_{\A} } \circ \iota_{\A, \Xaf\setminus\{M\} } \circ j_{ \afsy}( a_{2} )  \right)\overline{v}^{*} \\
&\qquad = \busby_{ \mathfrak{e}_{\B } } \circ \iota_{\B, \Xpi }\circ \psi   \circ j_{ \pisy } ( a_{1} ) + \overline{ \beta }_{ \{ M \} } \circ \busby_{ \mathfrak{e}_{\A} } \circ \iota_{\A, \Xaf\setminus\{M\} } \circ j_{ \afsy}( a_{2} ) \\
&\qquad =  \busby_{ \mathfrak{e}_{ \B } } \circ \iota_{\B, \Xpi } \circ \psi \circ j_{ \pisy }  ( a_{1} ) + \busby_{ \mathfrak{e}_{\B} } \circ \iota_{\B, \Xaf\setminus\{M\} } \circ \beta_{ \Xaf\setminus\{M\} } \circ j_{ \afsy} ( a_{2} ) \\
&\qquad = \busby_{ \mathfrak{e}_{\B} } \circ ( \psi \circ j_{\pisy }(a_{1}) + \beta_{ \Xaf\setminus\{M\} } \circ j_{ \afsy } ( a_{2} ) ). 
\end{align*} 
 Hence, 
\begin{equation}\label{eq}
\mathrm{Ad} ( \overline{v} ) \circ \overline{ \beta }_{ \{ M \} } \circ \busby_{ \mathfrak{e}_{\A} }  \circ ( \iota_{ \A, \Xpi } \circ j_{ \pisy } + \iota_{\A, \Xaf\setminus\{M\} } \circ j_{ \afsy} ) = \busby_{ \mathfrak{e}_{\B} } \circ ( \psi \circ j_{ \pisy } + \beta_{ \Xaf\setminus\{M\} } \circ j_{ \afsy } ). 
 \end{equation}
 
Note that the Busby invariant of the extension
\begin{equation*}
0 \to \A ( \{M\} )  \to \mathcal{E}_{1} \to ( p_{1} + p_{2} ) \left( \A ( \Xpi )  \oplus \A ( \Xaf\setminus\{M\} ) \right) ( p_{1} + p_{2} ) \to 0
\end{equation*}
is given by $\busby_{ \mathfrak{e}_{\A} }  \circ ( \iota_{ \A, \Xpi } \circ j_{ \pisy } + \iota_{ \A, \Xaf\setminus\{M\} } \circ j_{ \afsy} )$ and the Busby invariant of the extension
\begin{equation*}
0 \to \B ( \{ M \} )  \to \mathcal{E}_{2} \to ( e_{1} + e_{2} ) \left( \B ( \Xpi)  \oplus \B ( \Xaf\setminus\{M\} )  \right) ( e_{1} + e_{2} ) \to 0
\end{equation*}
is given by $\busby_{ \mathfrak{e}_{\B} } \circ ( \kappa_{\pisy} + \kappa_{ \afsy } )$, where $\ftn{ \kappa_{\pisy} }{ e_{1}  \B ( \Xpi ) e_{1} }{  \B ( \Xpi) }$ and $\ftn{ \kappa_{\afsy} }{ e_{2}  \B ( \Xaf\setminus\{M\}) e_{2} }{  \B ( \Xaf\setminus\{M\}) }$ are the natural embeddings.  Hence, by Equation (\ref{eq}), Theorem 2.2 of \cite{elp:mecpfbi}, and the five lemma, $\mathcal{E}_{1} \cong \mathcal{E}_{2}$.  By Lemma \ref{l:heralg}, $\mathcal{E}_{1}$ is isomorphic to a full hereditary sub-\Ca of $\A$ and $\mathcal{E}_{2}$ is isomorphic to a full hereditary sub-\Ca of $\B$.  We have just proved the claim.  

If $\Xpi=\emptyset$ the result is due to Elliott's classification result \cite{gae:cilssfa}, and if $\Xaf\setminus\{M\}=\emptyset$ the theorem follows easily by making modifications to the above proof. 
\end{proof}

\begin{remark}\label{r:class2}
Let \A and \B be graph \Cas satisfying Condition (K) that are \Cas over $X$ such that each of $\A(X_i),\B(X_i)$ are either \AFas or \Oia and such that $\A(X_i)$ and $\B(X_i)$ are tight \Cas over $X_i$, whenever $\A( X_{i} )$ and $\B( X_{i} )$ are \Oia.  Assume that there exists an isomorphism $\ftn{ \alpha }{ \FKplus{X}{\A} }{ \FKplus{X}{\B} }$. 
Assume moreover, that $\A ( \{ M \} )$ is an \AFa and that for every ideal $\IdealI$ of $\A$, we have that $\IdealI \subseteq \A ( \{ M \} )$ or $\A ( \{ M \} ) \subseteq \IdealI$.  Then $\A \otimes \K \cong \B \otimes \K$. 
This follows from the proof above together with Corollary~5.6 of \cite{segrer:ccfis} and applies to the cases\footnote{Here we specify how we view the algebras as algebras over $a\rightarrow b\leftarrow c$ by providing a continuous map from the primitive ideal space to $\{a,b,c\}$} 
\begin{enumerate}[(a)]
\item
\myref{4}{1E}{4} and \myref{4}{1E}{12}, 
where we view the algebra \A that is tight over the space \msp{4}{1E} as a \Ca over $a\rightarrow b \leftarrow c$ as indicated by the assignment 
\item
\myref{4}{1F}{4} and \myref{4}{1F}{12},
where we view the algebra \A that is tight over the space \msp{4}{1F} as a \Ca over $a\rightarrow b \leftarrow c$ as indicated by the assignment 
\end{enumerate}
\end{remark}

The following result resolves the cases \myref{3}{3}{2}, \myref{3}{3}{3}, \myref{4}{A}{1}, \myref{4}{A}{3}, \myref{4}{A}{7}. 

\begin{theorem}\label{t:class2-PI}
Let $\A$ and $\B$ be graph \Cas that are tight \Cas over $X$, with $X_i$ being a singleton, for each $i=1,2,\ldots,n$.  Suppose there exists an isomorphism $\ftn{ \alpha }{ \FKplus{X}{\A }}{ \FKplus{X}{\B }}$ such that $\alpha$ lifts to an invertible element in $\kk ( X ; \A , \B )$.  Then $\A \otimes \K \cong \B \otimes \K$.  
\end{theorem}

\begin{proof}
Note that we may assume that $\A$ and $\B$ are stable \Cas.  If $\A ( \{ M \} )$ is an \AFa, then the theorem follows from Theorem \ref{t:class2}.  Suppose $\A ( \{ M \} )$ is \Oia.  Then $\B ( \{ M \} )$ is \Oia.  Hence, by Proposition~\ref{p:full} and Lemma~\ref{l:directsumfull}, the extensions
\begin{align*}
0 \to \A ( \{ M \} ) \to \A  \to \A ( X \setminus \{ M \} ) \to 0, \\
0 \to \B ( \{ M \} ) \to \B  \to \B ( X \setminus \{ M \} ) \to 0
\end{align*}
are full extensions.  
The theorem now follows from the results of Theorem~4.6 of \cite{segrer:ccfis}.
\end{proof}

\section{A pullback technique}

The main idea of this section is to write the algebra as a pullback of extensions we can classify coherently. 
The problem is, that classification usually does not give us unique isomorphisms on the algebra level. 
But when the quotient is an \AFa we can in certain cases use that the \kk-class of the isomorphism is unique. 
The main idea here is similar to the main idea of Section~\ref{sec:fanspaces}. 

\begin{lemma}\label{l:lemma-isomorporchic-pullbacks}
For each $i=1,2$, let there be given \Cas $\A_i$, $\B_i$, and $\mathfrak{C}_i$ together with \shoms $\ftn{\alpha_i}{\A_i}{\mathfrak{C}_i}$ and  $\ftn{\beta_i}{\B_i}{\mathfrak{C}_i}$. Let $\mathfrak{P}_i$ denote the pullback of $\A_i$ and $\B_i$ along $\alpha_i$ and $\beta_i$, for each $i=1,2$. 

Assume that there are isomorphisms $\ftn{\varphi_\A}{\A_1}{\A_2}$,  $\ftn{\varphi_\B}{\B_1}{\B_2}$ and  $\ftn{\varphi_\mathfrak{C}}{\mathfrak{C}_1}{\mathfrak{C}_2}$, such that 
the following diagram commutes: 
$$\xymatrix{ \A_1\ar[d]^{\varphi_\A}\ar[r]^{\alpha_1} & \mathfrak{C}_1\ar[d]^{\varphi_{\mathfrak{C}}} & \B_1 \ar[d]^{\varphi_\B}\ar[l]_{\beta_1} \\
\A_2\ar[r]^{\alpha_2} & \mathfrak{C}_2 & \B_2. \ar[l]_{\beta_2} }$$
Then we get a canonically induced isomorphism from $\mathfrak{P}_1$ to $\mathfrak{P}_2$. 
\end{lemma}
\begin{proof}
The existence of the \shom from $\mathfrak{P}_1$ to $\mathfrak{P}_2$ follows from the universal property of the pullback. 
That this \shom is an isomorphism also follows from the universal property.
\end{proof}

\begin{lemma}\label{l:lemma-orthogonal-ideals}
Let $\IdealI$ and $\IdealJ$ be ideals of a \Ca \A satisfying $\IdealI\cap\IdealJ=0$. 
Then $\A$ is the pullback of $\A/\IdealJ$ and $\A/\IdealI$ along the quotient maps $\A/\IdealJ\to\A/(\IdealI+\IdealJ)$ and $\A/\IdealI\to\A/(\IdealI+\IdealJ)$. 
\end{lemma} 
\begin{proof}
This follows from Proposition~3.1 of \cite{gkp:pullpush} by noting that we have a commuting diagram 
$$\xymatrix{ \IdealI\ar@{=}[d]\ar[r] & \A\ar[d]\ar[r] & \A/\IdealI\ar[d] \\
\IdealI\ar[r] & \A/ \IdealJ \ar[r] & \A/(\IdealI+\IdealJ) }$$ 
with short exact rows. 
\end{proof}

%
%
%
%
%
%
%
%
%
%
%
%
%
%
%
%
%

The signatures \myref{4}{E}{4} and \myref{4}{E}{5} are covered by the following theorem. 

\begin{theorem} \label{t:pullback-technique}
Let \A and \B be graph \Cas that are tight over $X$, where $X$ is some finite $T_0$ space. 
Assume that there exists an isomorphism $\ftn{ \alpha }{ \FKplus{X}{\A} }{ \FKplus{X}{\B} }$. 
Assume, moreover, that we have disjoint open subsets $O_0$ and $O_1$ of $X$. 
Let 
$$Y_0=X\setminus O_1, \qquad Y_1=X\setminus O_0,\qquad\text{and}\qquad Z=X\setminus (O_0\cup O_1).$$
Assume also $Z\neq\emptyset$ and that $\A(Z)$ is an \AFa. 

For each $i=0,1$, if $\A(O_i)$ is \Oia, then we assume that:
\begin{enumerate}[(a)] 
\item 
There exist two disjoint clopen subsets $Y_i^1$ and $Y_i^2$ of $Y_i$ (with the subspace topology) such that $Y_i=Y_i^1\cup Y_i^2$ and $O_i\subseteq Y_i^1$. 
\item
The ideal lattice of $\A(O_i)$ is linear, i.e., $O_i\cong \X_j$ for some $j$.
\item
$\A(O_i)$ is an essential ideal of $\A(Y_i^1)$
\item 
$\A(\{m_i\})$ is essential in $\A(\{m_i\}\cup (Y_i^1\setminus O_i))$, 
where $m_i$ is the least element of $O_i$.  
\end{enumerate}

Then $\A \otimes \K \cong \B \otimes \K$. 
\end{theorem}
\begin{proof}
We may assume that \A and \B are stable \Cas. 
Note that for each locally closed subset $Y$ of $X$, $\A(Y)$ is an \AFa if and only if $\B(Y)$ is an \AFa, and $\A(Y)$ is \Oia if and only if $\B(Y)$ is \Oia (since there exists a positive isomorphism from $K_0(\A(Y))$ to $K_0(\B(Y))$). 
Specifically $\B(X\setminus(O_0\cup O_1))$ is an \AFa. 

Note that the diagram 
$$\xymatrix{ 0\ar[r]\ar[d] & \A(O_1)\ar[d]\ar@{=}[r] & \A(O_1)\ar[d] \\
\A(O_0)\ar[r]\ar@{=}[d] & \A \ar[d]\ar[r] & \A(Y_1)\ar[d] \\
\A(O_0)\ar[r] & \A(Y_0)\ar[r] & \A(Z) 
}$$
is commutative with short exact rows and columns, analogously for $\B$. 

If both $\A(O_0)$ and $\A(O_1)$ are \AFas, then it follows from the permanence properties of \AFas that \A is an \AFa, and thus also \B. 
In this case the theorem follows from Elliott's classification result \cite{gae:cilssfa}. 

Now assume that $\A(O_0)$ is an \AFa and that $\A(O_1)$ is \Oia. 
Let $Z_1^1=Z\setminus Y_1^2$ and $Z_1^2=Y_1^2$. Then $Z_1^1$ and $Z_1^2$ are locally closed subsets of $X$, and $Z$ is the disjoint union of $Z_1^1$ and $Z_1^2$.
Since $\A(Y_0)$ and $\B(Y_0)$ are extensions of \AFas, these are themselves \AFas. 
Since $\ftn{\alpha_{Y_0}}{K_0(\A(Y_0))}{K_0(\B(Y_0))}$ is a positive isomorphism, there exists an isomorphism $\ftn{\beta}{\A(Y_0)}{\B(Y_0)}$ such that $K_0(\beta)=\alpha_{Y_0}$ (by Elliott's classification result \cite{gae:cilssfa}). Since $\A ( Y_0 )$ and $\B ( Y_0 )$ are \AFas and $\beta$ is an $Y_0$-equivariant isomorphism, we have that $K_{0} ( \beta_Y ) = \alpha_{ Y }$ for all $Y \in \mathbb{LC} ( X )$ such that $Y \subseteq Y_0$.  

Let 
\begin{equation*}
\mathfrak{e}^\A \ : \ 0 \to \A ( O_1 ) \to \A( Y_1^1 ) \to \A ( Z^1 ) \to 0.
\end{equation*}
and
\begin{equation*}
\mathfrak{e}^\B \ : \ 0 \to \B ( O_1 ) \to \B( Y_1^1 ) \to \B ( Z^1 ) \to 0.
\end{equation*}
Since  $\ftn{ \alpha }{ \FKplus{X}{\A} }{ \FKplus{X}{\B} }$ is an isomorphism, we also have an isomorphism  $\ftn{ \alpha_{Y_1^1} }{ \FKplus{Y_1^1}{\A} }{ \FKplus{Y_1^1}{\B} }$. 
So by Theorem~4.14 of of \cite{rmrn:ctsfk}, Kirchberg \cite{ek:nkmkna}, and Theorem~3.3 of \cite{segrer:ccfis}, there exists an isomorphism $\ftn{\varphi}{\A (O_1)}{\B (O_1)}$ 
such that $K_*(\varphi)=\alpha_{O_1}$, 
and
\begin{equation*}
\left[\busby_{\mathfrak{e}^\B}\circ\beta_{O_{Z^1}}\right]=\left[\overline{\varphi}\circ\busby_{\mathfrak{e}^\A}\right]
\end{equation*}
in $\kk^1( \A( Z^1 ),\B( O_1 ) )$, since $\kk(\beta_{Z^1})$ is the unique lifting of $\alpha_{Z^1}$. 

As in the proof of Proposition~6.3 of \cite{segrer:ccfis}, Corollary~5.3 of \cite{segrer:ccfis} implies that $\busby_{\mathfrak{e}^\A}$ and $\busby_{\mathfrak{e}^\B}$ are full extensions, and thus also the extensions with Busby maps $\busby_{\mathfrak{e}^\B}\circ\beta_{Z^1}$ and $\overline{\varphi}\circ\busby_{\mathfrak{e}^\A}$ are full. 
Since the extensions are non-unital and $\B( O_1 )$ satisfies the corona factorization property, there exists a unitary $u\in\multialg{\B( O_1 ) }$ such that 
\begin{equation*}
\busby_{\mathfrak{e}^\B}\circ\beta_{Z^1}=\Ad(\overline{u})\circ\overline{\varphi}\circ\busby_{\mathfrak{e}^\A} 
\end{equation*}
where $\overline{u}$ is the image of $u$ in the corona algebra (this follows from \cite{gaedk:avbfat} and \cite{dkpwn:cfpaue}). 
Hence, by Theorem 2.2 of \cite{elp:mecpfbi}, there exists an isomorphism $\ftn{\eta}{\A(Y_1^1)}{\B(Y_1^1)}$ such that $(\Ad(\overline{u})\circ\varphi,\eta,\beta_{Z_1^1})$ is an isomorphism from $\mathfrak{e}^\A$ to $\mathfrak{e}^\B$. 

Since the extension
$$0\to \A(O_1)\to\A(Y_1)\to\A(Z)\to 0$$
is the direct sum of the extensions
$$0\to\A(O_1)\to\A(Y_1^1)\to\A(Z^1)\to 0$$ 
and 
$$0\to 0\to \A(Z^2)\to\A(Z^2)\to 0$$
and analogously for $\B$, we get an isomorphism from 
$0\to \A(O_1)\to\A(Y_1)\to\A(Z)\to 0$ to $0\to \B(O_1)\to\B(Y_1)\to\B(Z)\to 0$, which is equal to $\beta_{Z}$ on the quotient. 
Now the theorem follows from Lemma~\ref{l:lemma-orthogonal-ideals} and Lemma~\ref{l:lemma-isomorporchic-pullbacks}. 

Now assume instead that both $\IdealI$ and $\IdealJ$ are \Oia. 
The proof is similar to the case above. 
Instead of lifting $\ftn{\alpha_{Y_0}}{K_0(\A(Y_0))}{K_0(\B(Y_0))}$ to $\ftn{\beta}{\A(Y_0)}{\B(Y_0)}$ we just lift $\ftn{\alpha_Z}{K_0(\A(Z))}{K_0(\B(Z))}$ to $\ftn{\beta}{\A(Z)}{\B(Z)}$. 
Then we do as above first for the extensions corresponding to the relative open subset $O_0$ of $Y_0$ and then for the extensions corresponding to the relative open subset $O_1$ of $Y_1$. 
As above, the theorem then follows from Lemma~\ref{l:lemma-orthogonal-ideals} and Lemma~\ref{l:lemma-isomorporchic-pullbacks}. 
\end{proof}

\section{Ad hoc methods}\label{adhoc}

In this section we present arguments which resolve the classification question for some examples of tempered ideal spaces which are not covered by the general results above. Most of the results are based on knowing {strong} classification for smaller ideal spaces, as explained below. Our results of this nature, presented in \cite{segrer:scecc}, are of a rather limited scope, and require restrictions on the $K$-theory, requiring the $K$-groups to be finitely generated,  or even for the graph \Ca to be unital.  We will see this idea in use in a very clear form in the two open cases for three primitive ideals (cf.\ Section~\ref{threeopen}) and in more complicated four-point cases. 

Our starting point is

\begin{theorem}\label{t:adhoc}
Let $\A_{1}$ and $\A_{2}$ be graph \Cas that are tight \Cas over a finite $T_0$-space $X$ and let $U \in \mathbb{O} (X)$ be non-empty.  Let $\mathfrak{e}_{i}$ be the extension $0 \rightarrow \A_{i} (U) \otimes \K \rightarrow \A_{i} \otimes \K \rightarrow \A_{i}(X \setminus U ) \otimes \K \rightarrow 0$.  Suppose
\begin{enumerate}[(1)]
\item $\mathfrak{e}_{i}$ is a full extension;

\item there exists an invertible element $\alpha \in \kk (X; \A_{1} , \A_{2} )$; and

\item the induced invertible element $\alpha_{Y} \in \kk (\A_{1}(Y) \otimes \K , \A_{2}(Y) \otimes \K )$ lifts to an isomorphism from $\A_{1}(Y) \otimes \K$ to $\A_{2}(Y) \otimes \K$ for $Y = U$ and $Y = X \setminus U$.
\end{enumerate}
Then $\A_{1} \otimes \K \cong \A_{2} \otimes \K$. 
\end{theorem}

\begin{proof}
By (3), there exists an isomorphism $\varphi_{Y} : \A_{1}(Y) \otimes \K \rightarrow \A_{2}(Y) \otimes \K$ for $Y = U$ and $Y= X \setminus U$ such that $\kk ( \varphi_{Y} ) = \alpha_{Y}$.  It follows from (1) that $\frak{e}_i$ are  essential, so by \cite[Theorem~3.3]{segrer:ccfis}, $\alpha_{X\setminus U} \times [ \busby_{ \mathfrak{e}_{2} } ] = [ \busby_{ \mathfrak{e}_{1} } ] \times \alpha_{U}$.  Therefore, $\kk( \varphi_{X \setminus U } ) \times  [ \busby_{ \mathfrak{e}_{2} } ] = [ \busby_{ \mathfrak{e}_{1} } ] \times \kk(\varphi_{U} )$.  Hence, by \cite[Proposition~6.1 and Lemma~4.5]{segrer:ccfis}, we have that $\A_{1} \otimes \K \cong \A_{2} \otimes \K$. 
\end{proof}

\begin{definition}\label{def:class}
For a $T_{0}$ topological space $X$, we will consider classes $\mathcal{C}_{X}$ of separable, nuclear \Cas in the bootstrap category of Rosenberg and Schochet $\mathcal{N}$ such that 
\begin{enumerate}[(1)]
\item any element in $\mathcal{C}_{X}$ is a \Ca over $X$;

\item if $\A$ and $\B$ are in $\mathcal{C}_{X}$ and there exists an invertible element $\alpha$ in $\kk ( X ; \A, \B)$ which induces an isomorphism from $\FKplus{X}{ \A }$ to $\FKplus{X}{\B}$, then there exists an isomorphism $\ftn{ \varphi }{ \A }{ \B }$ such that $\kk ( \varphi ) = \alpha_{X}$, where $\alpha_{X}$ is the element in $\kk (\A, \B )$ induced by $\alpha$.
\end{enumerate} 
\end{definition}

\begin{remark}
Let $X$ be a finite $T_0$-space, let $U$ be an open subset of $X$, and let $\mathcal{C}_{U}$ and $\mathcal{C}_{X \setminus U}$ be classes of \Cas satisfying the conditions of Definition~\ref{def:class}.  If $\A_{1}$ and $\A_{2}$ are separable \Cas such that $\A_{1} (U), \A_{2} (U) \in \mathcal{C}_{U}$ and $\A_{1} (X \setminus U), \A_{2} (X \setminus U) \in \mathcal{C}_{X \setminus U}$, then (3) in Theorem~\ref{t:adhoc} holds.

Let $\mathcal{C}_{X}$ and $\mathcal{C}_{Y}$ be classes of \Cas satisfy the conditions in Definition~\ref{def:class}.  Let $\mathcal{C}_{X \sqcup Y}$ be the classes of \Cas consisting of elements $\A \oplus \B$ with $\A \in \mathcal{C}_{X}$ and $\B \in \mathcal{C}_{Y}$.  Then $\mathcal{C}_{X \sqcup Y}$ satisfies the conditions in Definition~\ref{def:class}.
\end{remark}

\begin{remark}
Here we will provide some examples of classes satisfying the conditions in Definition~\ref{def:class}.

\begin{enumerate}[(1)]
\item By \cite{ek:nkmkna}, the class all stable, nuclear, separable, \Oia \Cas that are tight over a finite $T_0$-space satisfy the conditions in Definition~\ref{def:class}.
\end{enumerate}


%
%

By \cite[Corollary~3.10 and Theorem~3.13]{segrer:scecc} and by the results of \cite{setk:spipgc}, the following classes of \Cas satisfies the conditions in Definition~\ref{def:class}.  

\begin{enumerate}[(1)]\addtocounter{enumi}{1}
\item Let $\mathcal{C}_{\X_{n}}$ be the class of nuclear, separable, tight \Cas $\A$ over $\X_{n}$ such that $\A$ is stable, $\A( \{ n\} )$ is a Kirchberg algebra, $\A( [ 1, n-1] )$ is an AF-algebra, and $K_{i} (\A[Y])$ is finitely generated for all $Y \in \mathbb{LC}(\X_{n})$.

\item Let $\mathcal{C}_{\X_{2}}'$ be the class of unital graph \Cas with exactly one non-trivial ideal with the ideal being an \AFa and the quotient \Oia, simple \Cas.  Let $\mathcal{C}_{\X_{2}}$ be the class of \Cas $\A$ such that $\A \cong \B \otimes \K$ for some $\B \in \mathcal{C}_{\X_{2}}'$. 
\end{enumerate}

By \cite{gae:cilssfa}, the following class of \Cas satisfy the conditions in Definition~\ref{def:class}.  

\begin{enumerate}[(1)]\addtocounter{enumi}{3}
\item Let $\mathcal{C}_{X}$ be the class of stable AF-algebras over $X$.
\end{enumerate}
\end{remark}

\subsection{Linear spaces}\label{threeopen}

This case is solved in \cite{segrer:scecc}, and the reader is referred there for details. However, since this is the most basic case in which our approach via Theorem \ref{t:adhoc} is applied, we will explain the methods for the benefit of the reader. 

\begin{lemma}\label{l:fullmixed}
Let $\A$ be a graph $C^*$-algebra such that $\A$ is a tight \Ca over $\X_{n}$.
\begin{enumerate}[(i)]
\item If $\A(\{n\})$ and $\A( \{ 1 \} )$ are \Oia and $\A ([ 2, n-1] )$ is an AF-algebra, then 
\begin{align*}
\mathfrak{e} : 0 \to \A ( [ 2, n ]) \otimes \K \to \A \otimes \K \to\A ( \{ 1 \} ) \otimes \K \to 0
\end{align*}
is a full extension.

\item If $\A ( [ k, n ] )$ and $\A ( [ 1, k-2 ] )$ are AF-algebras and $\A( \{ k-1\} )$ is \Oia, then 
\begin{align*}
\mathfrak{e} : 0 \to \A( [ k , n ] ) \otimes \K \to \A \otimes \K \to \A ( [ 1, k-1 ] ) \otimes \K \to 0
\end{align*}
is a full extension.

\item If $\A ( [ k, n ] )$ and $\A ( [ 1, k-2 ] )$ are AF-algebras and $\A( \{ k-1\} )$ is \Oia, then 
\begin{align*}
\mathfrak{e} : 0 \to \A( [ k -1, n ] ) \otimes \K \to \A \otimes \K \to \A ( [ 1, k-2 ] ) \otimes \K \to 0
\end{align*}
is a full extension.
\end{enumerate}
\end{lemma}

\begin{proof}
In \cite{segrer:scecc}, we prove (i) and (ii).  We now prove (iii).  Note that 
$$
0 \to \A( \{ k -1\} ) \otimes \K \to \A( [k-2,k-1] ) \otimes \K \to \A ( \{ k-2 \} ) \otimes \K \to 0
$$
is full since this is an essential extension and $\A( \{ k-1\} )$ is \Oia.  Since $\A( [ k,n] )$ is the largest \textit{AF}-ideal of $\A( [ k - 1, n ] )$ and $\A ( [ k-1,n] ) / \A( [ k , n ] ) = \A ( \{ k - 1 \} )$ is \Oia, by \cite[Proposition~3.10]{semt:cnga} and \cite[Lemma~1.5]{segrer:cecc}, $0 \to \A( [k ,n] ) \otimes \K \to \A( [k-1,n] ) \otimes \K \to \A ( \{ k -1\} ) \otimes \K \to 0$ is full.  By \cite[Proposition~3.2]{segrer:gclil}, $0 \to \A( [ k -1,n]  ) \otimes \K \to \A( [k-2,n] ) \otimes \K \to \A ( \{ k-2 \} ) \otimes \K \to 0$ is full.  Since $\A ( \{ k -2 \} ) = \A( [k-2,n] ) / \A( [ k -1,n]  )$ is an essential of $\A / \A ( [ k-1,n])$, the extension in (iii) is full by \cite[Proposition~5.4]{segrer:ccfis}. 
\end{proof}

%

To solve the cases \myref{3}{7}{5} and \myref{4}{3F}{9}, we now argue as follows:

\begin{theorem}\label{t:graphmixed1}
Let $\A_{1}$ and $\A_{2}$ be graph \Cas that are tight \Cas over $\X_{n}$.  Suppose
\begin{enumerate}[(i)]
\item $\A_{i}( \{ n \} )$ and $\A_{i} ( \{ 1 \} )$ are \Oia;

\item $\A_{i} ( [ 2, n-1] )$ is an AF-algebra; and

\item the $K$-groups of $\A_{i}$ are finitely generated.
\end{enumerate}
Then $\A_{1} \otimes \K \cong \A_{2} \otimes \K$ if and only if $\FKplus{\X_n}{ \A_{1} \otimes \K}  \cong \FKplus{\X_n}{ \A_{2} \otimes \K }$.
\end{theorem}

\begin{proof}
Let $\mathfrak{e}_{i}$ be the extension 
\begin{align*}
0 \to \A_{i} ( [ 2, n] ) \otimes \K \to \A_{i} \otimes \K \to \A_{i} ( \{ 1\}) \otimes \K \to 0.
\end{align*} 
By Lemma \ref{l:fullmixed}(i), $\mathfrak{e}_{i}$ is a full extension.  Thus, Assumption (1) of Theorem~\ref{t:adhoc} holds.  Suppose $\ftn{ \alpha }{ \FKplus{\X_n}{ \A_{1} \otimes \K } }{ \FKplus{\X_n}{ \A_{2} \otimes \K } }$ is an isomorphism.  Lift $\alpha$ to an invertible element $x \in \kk ( \X_n; \A_{1} \otimes \K , \A_{2} \otimes \K )$, such a lifting exists by Theorem~4.14 of \cite{rmrn:ctsfk}.  Therefore, Assumption (2) of Theorem~\ref{t:adhoc} holds.  

Note now that $x$ induces invertible elements $r_{\X_n}^{[2,n]} (x)$ in $\kk( [2,n] ; \A_{1} ( [ 2,n] ) \otimes \K , \A_{2} ( [2, n ] ) \otimes \K )$ and $r_{\X_n}^{[1]} (x)$ in $\kk ( \A ( \{ 1\} ) \otimes \K , \A_{2} ( \{ 1 \} ) \otimes \K )$. Note that $\A_{i} ([2,n])$ has a smallest ideal $\A_{i} ( \{ n \} )$ which is \Oia and the quotient $\A_{i} ( [2,n-1])$ is an \AFa.  By Theorem~3.9 of \cite{segrer:scecc}, there exists an isomorphism $\ftn{ \varphi }{ \A_{1}( [ 2, n ]) \otimes \K }{ \A_{2} ( [ 2, n ] ) \otimes \K }$ such that $\kl ( \varphi )$ is the invertible element in $\kl ( \A_{1} ( [2,n] ) , \A_{2} ( [2,n] ) )$ induced by $x$.  Since the $K$-theory of $\A_{1}$ is finitely generated, $\kl ( \A_{1} ( [2,n] ) , \A_{2} ( [2,n] ) ) = \kk ( \A_{1} ( [2,n] ) , \A_{2} ( [2,n] ) )$.  Thus,  $\kk ( \varphi )$ is the invertible element in $\kk ( \A_{1} ( [2,n] ) , \A_{2} ( [2,n] ) )$ induced by $x$.  By the Kirchberg-Phillips classification, there exists an isomorphism $\ftn{ \psi }{ \A_{1} (\{1 \} ) \otimes \K }{ \A_{2} (\{ 1\}) \otimes \K }$ lifting $r_{\X_n}^{[1]} (x)$.  We have just shown that Assumption (3) of Theorem~\ref{t:adhoc} holds.

By Theorem~\ref{t:adhoc}, we can conclude that $\A_{1} \otimes \K \cong \A_{2} \otimes \K$.
\end{proof}

Similarly, one solves \myref{3}{7}{2}, \myref{4}{3F}{2}, and \myref{4}{3F}{4} using

\begin{theorem}\label{t:graphmixed2}
Let $\A_{1}$ and $\A_{2}$ be graph \Cas that are tight \Cas over $\X_n$.  Suppose
\begin{enumerate}[(i)]
\item $\A_{i} ( [k,n] )$ and $\A_{i}  ( [ 1, k-2 ] )$ are \AFas{};
\item $\A_{i} ( \{ k-1\} )$ is \Oia; and
\item the $K$-groups of $\A_{i}$ are finitely generated.
\end{enumerate}
Then $\A_{1} \otimes \K \cong \A_{2} \otimes \K$ if and only if $\FKplus{\X_n}{ \A_{1}  \otimes \K } \cong \FKplus{\X_n}{ \A_{2} \otimes \K}$.  
\end{theorem}

A proof is given in \cite{segrer:scecc}.

\subsection{Accordion spaces}

\begin{lemma}
Let $\A$ be a graph \Ca with signature \mspk{4}{F}, and let $\IdealI$ be the smallest ideal of $\A$.
\begin{enumerate}[(1)]
\item When $\si{x} = \si{3}, \si{5}, \si{7},\si{9},\si{A},\si{B},\si{D}$, then the extension $0 \to \IdealI \otimes \bK \to \A \otimes \bK \to \A / \IdealI \otimes \bK  \to 0$ is full.

\item When $\si{x} =\si{2} , \si{4},\si{C}$, then the extension $0 \to \IdealI \otimes \bK \to \A \otimes \bK  \to \A / \IdealI \otimes \bK \to 0$ is full provided that $\A$ is unital.
\end{enumerate}
\end{lemma}

\begin{proof}
First note that the extension $0 \to \IdealI \otimes \K  \to \A \otimes \K \to \A / \IdealI \otimes \K \to 0$ is essential.  Hence, in the case \mspk{4}{F} for $\si{x} = \si{3}, \si{5}, \si{7},\si{9},\si{B},\si{D}$ the extension is full since $\IdealI \otimes \K$ is a simple, purely infinite, stable \Ca, which implies that $\corona{\IdealI \otimes \K}$ is simple.  If $A$ is unital and $Y$ is the space \mspk{4}{F} for $\si{ x}=\si{2}, \si{4}$, and $\si{C}$, then the extension is full since in this case $\IdealI \cong \K$ and $\corona{\K}$ is simple.  We are left with showing the extension is full for the case \myref{4}{F}{10}.  This case follows from \cite[Proposition~5.4 and Corollary~5.6]{segrer:ccfis}. 
\end{proof}

\begin{lemma}
Let $\A$ be a graph \Ca with tempered signature \mspk{4}{3F} for $\si{x} = \si{5},\si{6},\si{A},\si{D}$.
Then the ideal lattice of $\A$ is $0 \unlhd \IdealI_{1} \unlhd \IdealI_{2} \unlhd \IdealI_{3} \unlhd \A$ and the extension $0 \to \IdealI_{2} \otimes \K \to \A \otimes \K \to \A/ \IdealI_{2} \otimes \K \to 0$ is full.  
\end{lemma}

\begin{proof}
We will for show that $\mathfrak{e} : 0 \to \IdealI_{2} \otimes \K \to \IdealI_{3} \otimes \K \to \IdealI_{3} / \IdealI_{2} \otimes \K \to 0$ is a full extension.  By Lemma~\ref{l:fullmixed}, $\mathfrak{e}$ is a full extension for $\si{x} = \si{5},\si{A},\si{D}$.  Consider the case $\si{x} = \si{6}$.  Note that $\IdealI_{2}$ and $\IdealI_{3} / \IdealI_{1}$ are isomorphic to non-\textit{AF} graph \Cas with exactly one nontrivial ideal.  Therefore, by Proposition~\ref{p:full},
\begin{center}
$0 \to \IdealI_{1} \otimes \K \to \IdealI_{2} \otimes \K \to \IdealI_{2} / \IdealI_{1} \otimes \K \to 0$ \\
$0 \to \IdealI_{2} / \IdealI_{1} \otimes \K \to \IdealI_{3}/ \IdealI_{1} \otimes \K \to \IdealI_{3}/ \IdealI_{2} \otimes \K \to 0$
\end{center}
are full extensions.  By \cite[Proposition~3.2]{segrer:gclil}, $\mathfrak{e}$ is a full extension.  The lemma now follows from \cite[Proposition~5.4]{segrer:ccfis}.  
\end{proof}

\begin{lemma}
Let $\A$ be a graph \Ca with tempered signature  \mspk{4}{39} for $\si{x}=\si{2},\si{6},\si{9}, \si{A}, \si{B}, \si{C}, \si{D},$ or \si{E}.  Let $\IdealI$ be the greatest proper ideal of $\A$.
\begin{enumerate}[(1)]
\item If $\A$ is unital, then the extension $0 \to \IdealI \otimes \bK \to \A \otimes \bK  \to \A / \IdealI \otimes \bK \to 0$ is full.
\item When $\si{x} = \si{9}, \si{B}, \si{C}, \si{D}$, the extension $0 \to \IdealI \otimes \bK \to \A \otimes \bK  \to \A / \IdealI \otimes \bK \to 0$ is full.
\end{enumerate} 
\end{lemma}

\begin{proof}
  Suppose $\A$ is unital.  Using the general theory of graph \Cas with this specific ideal structure, we have that $\IdealI$ is stable.  Since $\A / \IdealI$ is simple and unital, the conclusion now follows from \cite[Lemma~1.5 and Proposition~1.6]{segrer:cecc}.  We now prove the extension $0 \to \IdealI \otimes \bK \to \A \otimes \bK  \to \A / \IdealI \otimes \bK \to 0$ is always full for the spaces \mspk{4}{39} with $\si{x} = \si{9}, \si{B}, \si{C}, \si{D}$.  Note that $\IdealI = \IdealI_{1} \oplus \IdealI_{2}$ with $\IdealI_{1}$ simple and $\IdealI_{2}$ a tight \Ca over $\X_{2}$.  By \cite[Lemma~4.5]{segrer:scecc} and \cite[Corollary~5.3 and Corollary~5.6]{segrer:ccfis}, we have $0 \to \IdealI_{2} \otimes \bK \to \A / \IdealI_{1} \otimes \K \to ( \A / \IdealI )  \otimes \K \to 0$ is full.  Since $\A / \IdealI_{2} \otimes \bK$ is a non-\textit{AF} graph \Ca with exactly one nontrivial ideal, the extension $0 \to \IdealI_{1} \otimes \bK \to \A / \IdealI_{2} \otimes \bK \to \A / \IdealI   \otimes \K \to 0$ is a full extension (cf.\ Proposition~\ref{p:full}).  Thus, by Lemma~\ref{l:full}, $0 \to \IdealI \otimes \bK \to \A \otimes \bK  \to \A / \IdealI \otimes \bK \to 0$ is full.
\end{proof}

Using the above lemmas and the Universal Coefficient Theorem of Bentmann and K{\"o}hler \cite{rbmk:uctcfts}, we get the following cases: 

\begin{corollary}
Let \A and \B be graph \Cas that are tight over a finite accordion space $X$. 
Assume that there exists an isomorphism from $\FKplus{X}{\A}$ to $\FKplus{X}{\B}$. 
If 
\begin{enumerate}[(1)]\label{adhocaccordion}
\item
\A and \B both have tempered signature \myref{4}{F}{7}, \myref{4}{F}{9}, \myref{4}{39}{11}, \myref{4}{39}{12}, or 
\item
\A and \B both have finitely generated $K$-theory and have tempered signature \myref{4}{F}{3}, \myref{4}{F}{10}, \myref{4}{F}{11}, \myref{4}{39}{9}, \myref{4}{39}{13}, \myref{4}{3F}{5}, \myref{4}{3F}{13},  or 
\item
\A and \B both are unital and have tempered signature \myref{4}{F}{2}, \myref{4}{F}{4}, \myref{4}{F}{5}, \myref{4}{F}{12}, \myref{4}{F}{13}, \myref{4}{39}{2}, \myref{4}{39}{6}, \myref{4}{39}{10}, \myref{4}{39}{14}, \myref{4}{3F}{6}, \myref{4}{3F}{10},  
\end{enumerate}
then $\A\otimes\K\cong\B\otimes\K$. 
\end{corollary}

\begin{proof}
By the above lemmas, all the extensions are full.  Note that the specified ideal and quotient for each space belongs to classes of \Cas satisfying the conditions in Definition~\ref{def:class}.  Hence, the result now follows from Theorem~\ref{t:adhoc} and the UCT for accordion spaces.
\end{proof}

\subsection{$Y$-shaped spaces}

%

\begin{lemma}
Let $\A$ be a graph \Ca with tempered signature  \mspk{4}{1F} for $\si{x} = \si{2},\si{5},\si{6},\si{7}$, or $\si{D}$, 
and let $\IdealI_{1}$ be the smallest ideal of $\A$ and let $\IdealI_{2}$ be the ideal of $\A$ containing $\IdealI_{1}$ such that $\IdealI_{2} / \IdealI_{1}$ is simple.
\begin{enumerate}[(1)]
\item When $\si{x} = \si{2},\si{6},\si{7}$, or $\si{D}$, the extension $0 \to \IdealI_{2} \otimes \bK  \to \A \otimes \bK  \to \A / \IdealI_{2} \otimes \bK \to 0$ is full.

\item When $\si{x}=\si{5}$, the extension $0 \to \IdealI_{2} \otimes \bK  \to \A \otimes \bK  \to \A / \IdealI_{2} \otimes \bK \to 0$ is full if $\A$ is unital.
\end{enumerate}
\end{lemma}

\begin{proof}
Let $\IdealJ_{1}$ and $\IdealJ_{2}$ be the maximal ideals of $\A$ containing $\IdealI_{2}$.  Suppose $\si{x} = \si{2},\si{6},\si{7}$, or $\si{D}$.    Then, by Lemma~\ref{l:fullmixed}, \cite[Proposition~3.2]{segrer:gclil}, and \cite[Corollary~5.3 and Corollary~5.6]{segrer:ccfis}, $0 \to \IdealI_{2} \otimes \bK \to \IdealJ_{\ell} \otimes \bK \to \IdealJ_{\ell} / \IdealI_{2}  \otimes \bK \to 0$ is full.  Hence, by Lemma~\ref{l:directsumfull}, $0 \to \IdealI_{2} \otimes \bK  \to \A \otimes \bK  \to \A / \IdealI_{2} \otimes \bK \to 0$  is full.  

Suppose that the signature is \myref{4}{1F}{5} and $\A$ is unital.  Assume that $\IdealJ_{1} / \IdealI_{2}$ is an AF-algebra and  $\IdealJ_{2} / \IdealI_{2}$ is purely infinite.  By Lemma~\ref{l:fullmixed}, $0 \to \IdealI_{2} \otimes \bK \to \IdealJ_{2} \otimes \bK \to \IdealJ_{2} / \IdealI_{2}  \otimes \bK \to 0$ is full.  Since $\A$ is a unital graph \Ca, we have that $\IdealI_{2} / \IdealI_{1} \cong \K$.  Therefore, $0 \to \IdealI_{2} / \IdealI_{1} \otimes \K \to \IdealJ_{1} / \IdealI_{1} \otimes \bK \to \IdealJ_{1} / \IdealI_{2}  \otimes \bK \to 0$ is full.  Since $\IdealI_{2}$ is a stably isomorphic to a non-AF graph \Ca with exactly one nontrivial ideal, by Proposition~\ref{p:full}, $0 \to \IdealI_{1} \otimes \K \to \IdealI_{2} \otimes \K \to \IdealI_{2} / \IdealI_{1} \otimes \K \to 0$ is full.  By \cite[Proposition~3.2]{segrer:gclil}, 
$$
0 \to \IdealI_{2} \otimes \bK \to \IdealJ_{1} \otimes \bK \to \IdealJ_{1} / \IdealI_{2}  \otimes \bK \to 0$$
 is full.  Hence, by Lemma~\ref{l:directsumfull}, $0 \to \IdealI_{2} \otimes \bK  \to \A \otimes \bK  \to \A / \IdealI_{2} \otimes \bK \to 0$  is full.

%
%
%
\end{proof}

\begin{lemma}
Let $\A$ be a graph \Ca with tempered signature \mspk{4}{3E} for $\si{x}= \si{3},\si{4},\si{5},\si{9}, \si{B}$, or \si{D}, 
and let $\IdealI_{1}$ and $\IdealI_{2}$ be the minimal ideals of $\A$.
\begin{enumerate}[(1)]
\item When $\si{x}=\si{3}, \si{4},\si{5},\si{B},\si{D}$, the extension $0 \to (\IdealI_{1} \oplus \IdealI_{2} ) \otimes \bK \to \A \otimes \K  \to \A / (\IdealI_{1} \oplus \IdealI_{2} ) \otimes \bK \to 0$ is a full extension.
\item When $\si{x} = \si{9}$, and  $\A$ is unital, then $0 \to (\IdealI_{1} \oplus \IdealI_{2} ) \otimes \bK \to \A \otimes \K \to \A / (\IdealI_{1} \oplus \IdealI_{2} ) \otimes \bK \to 0$ is a full extension.
\end{enumerate}
\end{lemma}

\begin{proof}
Suppose $\si{x} = \si{4},\si{5}, \si{B}$, or \si{D}.  Let $\IdealI$ be the ideal of $\A$ containing $(\IdealI_{1} \oplus \IdealI_{2} )$ such that $\IdealI / (\IdealI_{1} \oplus \IdealI_{2} ) $ is simple.  Note that push forward extension of the extension $0 \to (\IdealI_{1} \oplus \IdealI_{2} )  \otimes \bK \to \IdealI \otimes \bK \to \IdealI/ (\IdealI_{1} \oplus \IdealI_{2} )  \otimes \bK \to 0$ via the coordinate projection $(\IdealI_{1} \oplus \IdealI_{2} ) \rightarrow \IdealI_{i}$ is a full extension since its isomorphic to a non-\textit{AF} graph \Cas with exactly one nontrivial ideal.  Therefore, by Lemma~\ref{l:full}, $0 \to (\IdealI_{1} \oplus \IdealI_{2} )  \otimes \bK \to \IdealI \otimes \bK \to \IdealI/ (\IdealI_{1} \oplus \IdealI_{2} )  \otimes \bK \to 0$ is a full extension.  By \cite[Proposition~5.4]{segrer:ccfis}, $0 \to (\IdealI_{1} \oplus \IdealI_{2} ) \otimes \bK \to \A \otimes \bK \to \A / (\IdealI_{1} \oplus \IdealI_{2} ) \otimes \bK \to 0$ is a full extension since $\IdealI / (\IdealI_{1} \oplus \IdealI_{2}) \otimes \K$ is an essential ideal of $\A / (\IdealI_{1} \oplus \IdealI_{2}) \otimes \K$.  

We now prove the extension is full for the case $\si{x}=\si{3}$.  Note that in this case $\IdealI_{1} \otimes \K$ and $\IdealI_{2} \otimes \K$ are purely infinite, simple \Cas.  Let $\IdealI$ be the ideal of $\A$ containing $(\IdealI_{1} \oplus \IdealI_{2} )$ such that $\IdealI / (\IdealI_{1} \oplus \IdealI_{2} ) $ is simple.  By Lemma~\ref{l:essential} and Lemma~\ref{l:full}, $0 \to (\IdealI_{1} \oplus \IdealI_{2}) \otimes \K \to \IdealI \otimes \K \to \IdealI / (\IdealI_{1} \oplus \IdealI_{2}) \otimes \K \to 0$ is a full extension.  The conclusion now follows from \cite[Proposition~5.4]{segrer:ccfis} since $\IdealI / (\IdealI_{1} \oplus \IdealI_{2}) \otimes \K$ is an essential ideal of $\A / (\IdealI_{1} \oplus \IdealI_{2}) \otimes \K$.

Suppose $\si{x}=\si{9}$ and $\A$ is unital.  Then $\IdealI_{i}$ is either $\K$ or a stable, purely infinite, simple \Ca.  Let $\IdealI$ be the ideal containing $\IdealI_{1} \oplus \IdealI_{2}$ such that $\IdealI / (\IdealI_{1} \oplus \IdealI_{2} ) $ is simple.  Note that the signature of $\IdealI$ is $\msp{3}{6}$.  By Lemma~\ref{l:essential}, the push forward extension of the extension $0 \to (\IdealI_{1} \oplus \IdealI_{2} )  \otimes \bK \to \IdealI \otimes \bK \to \IdealI/ (\IdealI_{1} \oplus \IdealI_{2} )  \otimes \bK \to 0$ via the coordinate projection $(\IdealI_{1} \oplus \IdealI_{2} ) \otimes \K \rightarrow \IdealI_{i} \otimes \K$ is essential, and hence full since $\corona{ \IdealI_{i} \otimes \K}$ is simple.  Thus, by Lemma~\ref{l:full}, $0 \to (\IdealI_{1} \oplus \IdealI_{2} ) \otimes \bK \to \IdealI \otimes \bK \to  \IdealI / (\IdealI_{1} \oplus \IdealI_{2} ) \otimes \bK \to 0$ is full.  By \cite[Proposition~5.4]{segrer:ccfis}, $0 \to (\IdealI_{1} \oplus \IdealI_{2} ) \otimes \bK \to \A \otimes \bK \to  \A / (\IdealI_{1} \oplus \IdealI_{2} ) \otimes \bK \to 0$ is a full extension since $\IdealI / (\IdealI_{1} \oplus \IdealI_{2} )$ is an essential ideal of $\A/ (\IdealI_{1} \oplus \IdealI_{2} )$.
\end{proof}

\begin{lemma}
Let $\A$ be a graph \Ca with tempered signature \myref{4}{3E}{7}.  Let $\IdealI$ be the ideal of \A such that $\A / \IdealI$ is simple.  Then $0 \to \IdealI \otimes \K \to \A \otimes \K \to  \A / \IdealI \otimes \K \to 0$ is a full extension.
\end{lemma}

\begin{proof}
Let $\IdealI_{1}$ and $\IdealI_{2}$ be the minimal ideals of \A which is contained in $\IdealI$.  Since $\IdealI / ( \IdealI_{1} + \IdealI_{2} )$ is a non-unital, purely infinite, simple \Ca, we have that $0 \to \IdealI / ( \IdealI_{1} + \IdealI_{2} )  \otimes \K \to \A / ( \IdealI_{1} + \IdealI_{2} ) \otimes \K \to  \A / \IdealI \otimes \K \to 0$ is a full extension.  The conclusion of the lemma now follows from Corollary~5.3 of \cite{segrer:ccfis}.
\end{proof}

\begin{lemma}
Let $\A$ be a graph \Ca with tempered signature \myref{4}{1F}{14}.  Let $\IdealI$ be the smallest ideal of \A.  Then $0 \to \IdealI \otimes \K \to \A \otimes \K \to  \A / \IdealI \otimes \K \to 0$ is a full extension.
\end{lemma}

\begin{proof}
Let $\IdealI_{1}$ be the ideal of \A such that $\IdealI_{1}$ contains $\IdealI$ and $\IdealI_{1} / \IdealI$ is simple.  Since $\IdealI_{1}$ is stably isomorphic to a non-\textit{AF} graph \Ca with exactly one nontrivial ideal, we have that $0 \to \IdealI \otimes \K \to \IdealI_{1} \otimes \K \to \IdealI_{1} / \IdealI \otimes \K \to 0$ is full.  Since $\IdealI_{1} / \IdealI$ is an essential ideal of $\A / \IdealI$, the conclusion of the lemma follows from Proposition~5.4 of \cite{segrer:ccfis}.
\end{proof}

Using the above lemmas and the results of \cite{sagrer:fkrrzc}, we get the following: 

\begin{corollary}\label{adhocY}
Let \A and \B be graph \Cas with signature either \msp{4}{1F} or \msp{4}{3E}, and 
assume that there exists an isomorphism from $\FKplus{X}{\A}$ to $\FKplus{X}{\B}$. 
If 
\begin{enumerate}[(1)]
\item
\A and \B both have tempered signature \myref{4}{1F}{7}, \myref{4}{1F}{14}, \myref{4}{3E}{3}, \myref{4}{3E}{7}, or \myref{4}{3E}{13}, or
\item
\A and \B both have finitely generated $K$-theory and have tempered signature \myref{4}{1F}{13}, \myref{4}{3E}{4} or \myref{4}{3E}{5}, or 
\item
\A and \B both are unital and have tempered signature \myref{4}{1F}{2}, \myref{4}{1F}{5}, \myref{4}{1F}{6}, \myref{4}{3E}{9} or \myref{4}{3E}{11}, 
\end{enumerate}
then $\A\otimes\K\cong\B\otimes\K$. 
\end{corollary}

\begin{proof}
By the above lemmas, all the extensions are full.  Note that the specified ideal and quotient for each space belongs to classes of \Cas satisfying the conditions in Definition~\ref{def:class}.  Hence, the result now follows from Theorem~\ref{t:adhoc}.
\end{proof}

\subsection{$O$-shaped spaces}

\begin{lemma}
Let $\A$ be a graph \Ca that is a tight \Ca over the $O$-shaped space \myref{4}{3B}{7}.  Let $\IdealI$ be the smallest ideal of $\A$ and let $\IdealI_{1}$ and $\IdealI_{2}$ be the ideals of $\A$ which contain $\IdealI$ and $\IdealI_{k} / \IdealI$ is simple.  Then $0 \to (\IdealI_{1} + \IdealI_{2}) \otimes \K \to \A \otimes \K \to \A / (\IdealI_{1} + \IdealI_{2}) \otimes \K \to 0$ is a full extension. 
\end{lemma}

\begin{proof}
Note that $\A / \IdealI$ is a tight \Ca over the space 3.6.5.  Then by Lemma~\ref{l:full}, $0 \to (\IdealI_{1} + \IdealI_{2}) / \IdealI \otimes \K \to \A/ \IdealI \otimes \K \to \A/ (\IdealI_{1} + \IdealI_{2}) \otimes \K \to 0$ is a full extension since $\IdealI_{1} / \IdealI$ and $\IdealI_{2} / \IdealI$ are purely infinite, simple \Cas.  Also, since $\IdealI$ is an essential ideal of $\IdealI_{1} + \IdealI_{2}$ and since $\IdealI$ is a purely infinite, simple \Ca, we have that $0 \to \IdealI \otimes \K \to ( \IdealI_{1} + \IdealI_{2} ) \otimes \K \to ( \IdealI_{1} + \IdealI_{2} ) / \IdealI \otimes \K \to 0$ is a full extension.  The conclusion of the lemma now follows from Proposition~3.2 of \cite{segrer:gclil} since $\A / (\IdealI_{1} + \IdealI_{2})$ is simple.
\end{proof}

\begin{lemma}
Let $\A$ be a graph \Ca that is a tight \Ca over the $O$-shaped space \myref{4}{3B}{14}.  Let $\IdealI$ be the smallest ideal of $\A$.  Then $0 \to \IdealI \otimes \K \to \A \otimes \K \to \A / \IdealI \otimes \K \to 0$ is a full extension. 
\end{lemma}

\begin{proof}
Let $\IdealI_{1}$ and $\IdealI_{2}$ be the ideals of $\A$ which contain $\IdealI$ and $\IdealI_{k} / \IdealI$ is simple.  Since $\IdealI_{k} \otimes \K$ is isomorphic to a graph \Ca with exactly one non-trivial ideal and $\IdealI_{k} \otimes \K$ is not an \AFa, by Proposition~\ref{p:full}, we have that $0 \to \IdealI \otimes \K \to \IdealI_{k} \otimes \K \to \IdealI_{k} / \IdealI \otimes \K \to 0$ is a full extension.  By Lemma~\ref{l:directsumfull}, $0 \to \IdealI \otimes \K \to ( \IdealI_{1} + \IdealI_{2} ) \otimes \K \to ( \IdealI_{1} + \IdealI_{2} ) / \IdealI \otimes \K \to 0$ is a full extension.  The conclusion of the lemma now follows from Proposition~5.4 of \cite{segrer:cecc} since $( \IdealI_{1} + \IdealI_{2} ) / \IdealI \otimes \K$ is an essential ideal of $\A / \IdealI$.
\end{proof}

Using the above lemmas and the results of \cite{sarbtk:rfkccka}, we get the following cases:

\begin{corollary}\label{adhocO}
Let \A and \B be graph \Cas that are tight over a $O$-shaped space $X$.  Assume that there exists an isomorphism from $\FKplus{X}{\A}$ to $\FKplus{X}{\B}$.  If \A and \B both have tempered signature \myref{4}{3B}{7} or \myref{4}{3B}{14}, then $\A\otimes\K\cong\B\otimes\K$. 
\end{corollary}

\begin{proof}
By the above lemmas, all the extensions are full.  Note that the specified ideal and quotient for each space belongs to classes of \Cas satisfying the conditions in Definition~\ref{def:class}.  Hence, the result now follows from Theorem~\ref{t:adhoc}.
\end{proof}

\section{Summary of results}\label{resultsstart}

In this final section, we index our results. Cases that open are indicated by ``?''. Cases that are solved in general are marked by ``$\surd$'', and if we need to impose conditions of finitely generated $K$-theory or unitality, this is indicated by ``$\surd_{f.g.}$'' or  ``$\surd_{\mathbf 1}$'', respectively. 

\subsection{One point spaces}

Having nothing new to add, we include the simple case only for completeness. 

\begin{center}
\begin{tabular}{|c|c|c|c|}\hline
\multicolumn{4}{|c|}{\mspk{1}{0}}\\\hline \hline
\mylab{1}{0}{0}&$\square$&$\surd$ &\AFc\\\hline
\mylab{1}{0}{1}&$\blacksquare$&$\surd$&\PIc\\\hline
\end{tabular}
\end{center}

\subsection{Two point spaces}

This case was solved in \cite{semt:cnga}, so again we include it only for completeness.

\begin{center}
\begin{tabular}{|c|c|c|c|}\hline
\multicolumn{4}{|c|}{\mspk{2}{1}}\\\hline \hline
\mylab{2}{1}{0}&\twop{0}&$\surd$ &\AFc\\\hline
\mylab{2}{1}{1}&\twop{1}&$\surd$&\ETPIideal\\\hline
\mylab{2}{1}{2}&\twop{2}&$\surd$ &\ETAFideal\\\hline
\mylab{2}{1}{3}&\twop{3}&$\surd$&\PIc\\\hline
\end{tabular}
\end{center}

\subsection{Three point spaces}

We resolve the case of three primitive ideal spaces here, up to a condition of finite generation which must be imposed in the cases of signature \myref{3}{7}{2} and 
\myref{3}{7}{5}. We do not know if this condition is necessary. 

\begin{center}
\begin{tabular}{|c|c|c|c|c|c|c|c|c|}\cline{1-4}\cline{6-9}
\multicolumn{4}{|c|}{\mspk{3}{3}}&\qquad&\multicolumn{4}{|c|}{\mspk{3}{6}}\\
\cline{1-4}\cline{6-9}
\mylab{3}{3}{0}&\threepin{0}&$\surd$&\AFc    && \mylab{3}{6}{0}&\threepout{0}& $\surd$ &\AFc\\\cline{1-4}\cline{6-9}
\mylab{3}{3}{1}&\threepin{1}&$\surd$&\fandownAFideal&& \mylab{3}{6}{1}&\threepout{1}& $\surd$ &\fanupAFquotient \\\cline{1-4}\cline{6-9}
\mylab{3}{3}{2}&\threepin{2}&$\surd$&\fandownPIideal&& \mylab{3}{6}{2}&\threepout{2}& $\surd$ &\fanupPIquotient \\\cline{1-4}\cline{6-9}
\mylab{3}{3}{3}&\threepin{3}&$\surd$&\fandownPIideal&& \mylab{3}{6}{3}&\threepout{3}& $\surd$ &\fanupPIquotient \\\cline{1-4}\cline{6-9}
\mylab{3}{3}{5}&\threepin{5}&$\surd$&\fandownAFideal&& \mylab{3}{6}{5}&\threepout{5}& $\surd$ &\fanupAFquotient \\\cline{1-4}\cline{6-9}
\mylab{3}{3}{7}&\threepin{7}&$\surd$&\PIc    && \mylab{3}{6}{7}&\threepout{7}& $\surd$ &\PIc \\\cline{1-4}\cline{6-9}
\end{tabular}
\end{center}

\begin{center}
\begin{tabular}{|c|c|c|c|}\hline
\multicolumn{4}{|c|}{\mspk{3}{7}}\\\hline \hline
\mylab{3}{7}{0}&\threeplin{0}&$\surd$ &\AFc\\\hline
\mylab{3}{7}{1}&\threeplin{1}&$\surd$&\FININFPIideal\\\hline
\mylab{3}{7}{2}&\threeplin{2}&$\surd_{\fg}$ &\tref{t:graphmixed2} \\\hline
\mylab{3}{7}{3}&\threeplin{3}&$\surd$&\FININFPIideal\\\hline
\mylab{3}{7}{4}&\threeplin{4}&$\surd$&\FININFAFideal\\\hline
\mylab{3}{7}{5}&\threeplin{5}&$\surd_{\fg}$&\tref{t:graphmixed1} \\\hline
\mylab{3}{7}{6}&\threeplin{6}&$\surd$ &\FININFAFideal\\\hline
\mylab{3}{7}{7}&\threeplin{7}&$\surd$&\PIc\\\hline
\end{tabular}
\end{center}

\subsection{Four point spaces}\label{resultsend}

In this section, we present our results for the case of four primitive ideals. As will be obvious below, the strength of our results varies dramatically with the nature of the 
spaces. In general, we can say quite a lot about all spaces apart from
\msp{4}{E},  \msp{4}{1E}, and \msp{4}{3B}. It may be interesting to
note what makes these spaces difficult to handle; indeed the case
\msp{4}{E} is an accordion space in which a general UCT is know to
hold, but it differs from the other accordion  spaces by having poor
separation properties when it comes to establishing fullness. The
$O$-shaped spaces are also hard to separate fully, but have the added
difficulty that no general UCT is known for them.
\clearpage


\begin{sideways}
\begin{tabular}{|c|c|c|c|c|c|c|c|c|}\cline{1-4}\cline{6-9}
\multicolumn{4}{|c|}{\mspk{4}{E}}&\qquad&\multicolumn{4}{|c|}{\mspk{4}{F}}\\
\cline{1-4}\cline{6-9}
\mylab{4}{E}{0}&\fourE{0}&$\surd$ &\AFc && \mylab{4}{F}{0}&\fourF{0}&$\surd$&\AFc\\\cline{1-4}\cline{6-9}
\mylab{4}{E}{1}&\fourE{1}&$\surd$ &\rref{r:class1}&& \mylab{4}{F}{1}&\fourF{1}&$\surd$ &\FININFPIideal \\\cline{1-4}\cline{6-9}
\mylab{4}{E}{2}&\fourE{2}&?&&& \mylab{4}{F}{2}&\fourF{2}&$\surd_{\unital}$ & \cref{adhocaccordion}  \\\cline{1-4}\cline{6-9}
\mylab{4}{E}{3}&\fourE{3}&?&&& \mylab{4}{F}{3}&\fourF{3}& $\surd_{\fg}$ &\cref{adhocaccordion} \\\cline{1-4}\cline{6-9}
\mylab{4}{E}{4}&\fourE{4}&$\surd$ &\tref{t:pullback-technique} && \mylab{4}{F}{4}&\fourF{4}&$\surd_{\unital}$ &\cref{adhocaccordion} \\\cline{1-4}\cline{6-9}
\mylab{4}{E}{5}&\fourE{5}&$\surd$ &\tref{t:pullback-technique} && \mylab{4}{F}{5}&\fourF{5}&$\surd_{\unital}$&\cref{adhocaccordion} \\\cline{1-4}\cline{6-9}
\mylab{4}{E}{6}&\fourE{6}&?&&& \mylab{4}{F}{6}&\fourF{6}&$\surd$ & \tref{t:class2} \\\cline{1-4}\cline{6-9}
\mylab{4}{E}{7}&\fourE{7}& ?  &   && \mylab{4}{F}{7}&\fourF{7}&$\surd$ &\cref{adhocaccordion}\\\cline{1-4}\cline{6-9}
\mylab{4}{E}{8}&\fourE{8}&?&&& \mylab{4}{F}{8}&\fourF{8}&$\surd$ & \tref{t:class2} \\\cline{1-4}\cline{6-9}
\mylab{4}{E}{9}&\fourE{9}&?&&& \mylab{4}{F}{9}&\fourF{9}&$\surd$ &\cref{adhocaccordion}\\\cline{1-4}\cline{6-9}
\mylab{4}{E}{10}&\fourE{10}&?&&& \mylab{4}{F}{10}&\fourF{10}&$\surd_{\fg}$ & \cref{adhocaccordion} \\\cline{1-4}\cline{6-9}
\mylab{4}{E}{11}&\fourE{11}&?&&& \mylab{4}{F}{11}&\fourF{11}&$\surd_{\fg}$&\cref{adhocaccordion}\\\cline{1-4}\cline{6-9}
\mylab{4}{E}{12}&\fourE{12}&?&&& \mylab{4}{F}{12}&\fourF{12}&$\surd_{\unital}$ & \cref{adhocaccordion} \\\cline{1-4}\cline{6-9}
\mylab{4}{E}{13}&\fourE{13}& ?&  && \mylab{4}{F}{13}&\fourF{13}&$\surd_{\unital}$ &\cref{adhocaccordion} \\\cline{1-4}\cline{6-9}
\mylab{4}{E}{14}&\fourE{14}&?&&& \mylab{4}{F}{14}&\fourF{14}&$\surd$ &\tref{t:class2} \\\cline{1-4}\cline{6-9}
\mylab{4}{E}{15}&\fourE{15}&$\surd$ &\PIc && \mylab{4}{F}{15}&\fourF{15}&$\surd$ &\PIc\\\cline{1-4}\cline{6-9}
\end{tabular}
\end{sideways}

\begin{sideways}
\begin{tabular}{|c|c|c|c|c|c|c|c|c|}\cline{1-4}\cline{6-9} 
\multicolumn{4}{|c|}{\mspk{4}{39}}&\qquad&\multicolumn{4}{|c|}{\mspk{4}{3F}}\\
\cline{1-4}\cline{6-9}
\mylab{4}{39}{0}&\fourthreenine{0}&$\surd$ &\AFc&&                                         
\mylab{4}{3F}{0}&\fourthreeF{0}&$\surd$& \AFc \\\cline{1-4}\cline{6-9}                                       
\mylab{4}{39}{1}&\fourthreenine{1}&$\surd$ &\fanupAFquotient &&                               
\mylab{4}{3F}{1}&\fourthreeF{1}&$\surd$& \FININFPIideal \\\cline{1-4}\cline{6-9}                       
\mylab{4}{39}{2}&\fourthreenine{2}& $\surd_{\unital}$ &\cref{adhocaccordion}  &&                                 
\mylab{4}{3F}{2}&\fourthreeF{2}&$\surd_{\fg}$& \tref{t:graphmixed2} \\\cline{1-4}\cline{6-9}                      
\mylab{4}{39}{3}&\fourthreenine{3}&$\surd$ &\fanupAFquotient &&                               
\mylab{4}{3F}{3}&\fourthreeF{3}&$\surd$& \FININFPIideal\\\cline{1-4}\cline{6-9}                              
\mylab{4}{39}{4}&\fourthreenine{4}&$\surd$ &\fanupAFquotient  &&                                
\mylab{4}{3F}{4}&\fourthreeF{4}&$\surd_{\fg}$ & \tref{t:graphmixed2} \\\cline{1-4}\cline{6-9}                                    
\mylab{4}{39}{5}&\fourthreenine{5}&$\surd$ &\fanupAFquotient  &&                        
\mylab{4}{3F}{5}&\fourthreeF{5}&$\surd_{\fg}$& \cref{adhocaccordion} \\\cline{1-4}\cline{6-9}                       
\mylab{4}{39}{6}&\fourthreenine{6}& $\surd_{\unital}$ &\cref{adhocaccordion} &&                   
\mylab{4}{3F}{6}&\fourthreeF{6}&$\surd_{\unital}$ & \cref{adhocaccordion} \\\cline{1-4}\cline{6-9}                       
\mylab{4}{39}{7}&\fourthreenine{7}&$\surd$& \tref{t:class1}  &&                                      
\mylab{4}{3F}{7}&\fourthreeF{7}& $\surd$& \FININFPIideal \\\cline{1-4}\cline{6-9}                                 
\mylab{4}{39}{8}&\fourthreenine{8}&$\surd$&\ETAFideal&&                                        
\mylab{4}{3F}{8}&\fourthreeF{8}&$\surd$& \ETAFideal \\\cline{1-4}\cline{6-9}                                      
\mylab{4}{39}{9}&\fourthreenine{9}&$\surd_{\fg}$ &\cref{adhocaccordion}  &&                                   
\mylab{4}{3F}{9}&\fourthreeF{9}&$\surd_{\fg}$& \tref{t:graphmixed1} \\\cline{1-4}\cline{6-9}                           
\mylab{4}{39}{10}&\fourthreenine{10}&$\surd_{\unital}$& \cref{adhocaccordion}  &&                  
\mylab{4}{3F}{10}&\fourthreeF{10}&$\surd_{\unital}$& \cref{adhocaccordion} \\\cline{1-4}\cline{6-9}                     
\mylab{4}{39}{11}&\fourthreenine{11}&$\surd$ &\cref{adhocaccordion} &&                                                                     
\mylab{4}{3F}{11}&\fourthreeF{11}&?&  \\\cline{1-4}\cline{6-9}                 
\mylab{4}{39}{12}&\fourthreenine{12}& $\surd$ &\cref{adhocaccordion}  &&                                
\mylab{4}{3F}{12}&\fourthreeF{12}&$\surd$& \FININFAFideal \\\cline{1-4}\cline{6-9}                
\mylab{4}{39}{13}&\fourthreenine{13}&$\surd_{\fg}$& \cref{adhocaccordion}  &&                      
\mylab{4}{3F}{13}&\fourthreeF{13}&$\surd_{\fg}$& \cref{adhocaccordion} \\\cline{1-4}\cline{6-9}                         
\mylab{4}{39}{14}&\fourthreenine{14}& $\surd_{\unital}$ &\cref{adhocaccordion}  &&                               
\mylab{4}{3F}{14}&\fourthreeF{14}&$\surd$& \FININFAFideal \\\cline{1-4}\cline{6-9}                    
\mylab{4}{39}{15}&\fourthreenine{15}&$\surd$ &\PIc  &&                                               
\mylab{4}{3F}{15}&\fourthreeF{15}&$\surd$& \PIc \\\cline{1-4}\cline{6-9}                                   
\end{tabular}
\end{sideways}


\begin{center}                                                                                                        \begin{tabular}{|c|c|c|c|c|c|c|c|c|}\cline{1-4}\cline{6-9}                       
\multicolumn{4}{|c|}{\mspk{4}{A}}&\qquad&\multicolumn{4}{|c|}{\mspk{4}{38}}\\
\cline{1-4}\cline{6-9}                                                               
\mylab{4}{A}{0}&\fourA{0}&$\surd$& \AFc&&                                            \mylab{4}{38}{0}&\fourthreeeight{0}&$\surd$& \AFc \\\cline{1-4}\cline{6-9}                                                    
\mylab{4}{A}{1}&\fourA{1}&$\surd$& \fandownPIideal&&                                              \mylab{4}{38}{1}&\fourthreeeight{1}&$\surd$& \fanupAFquotient \\\cline{1-4}\cline{6-9}                                                    
\mylab{4}{A}{2}&\fourA{2}&$\surd$& \fandownAFideal &&                                             \mylab{4}{38}{3}&\fourthreeeight{3}&$\surd$& \fanupAFquotient \\\cline{1-4}\cline{6-9}                                                    
\mylab{4}{A}{3}&\fourA{3}&$\surd$& \fandownPIideal &&                                             \mylab{4}{38}{7}&\fourthreeeight{7}&$\surd$& \fanupAFquotient \\\cline{1-4}\cline{6-9}                                                    
\mylab{4}{A}{6}&\fourA{6}& $\surd$& \fandownAFideal &&                                            \mylab{4}{38}{8}&\fourthreeeight{8}&$\surd$& \fanupPIquotient \\\cline{1-4}\cline{6-9}                                                    
\mylab{4}{A}{7}&\fourA{7}&$\surd$& \fandownPIideal &&                                             \mylab{4}{38}{9}&\fourthreeeight{9}&$\surd$& \fanupPIquotient \\\cline{1-4}\cline{6-9}                                                    
\mylab{4}{A}{14}&\fourA{14}&$\surd$& \fandownAFideal &&                                           \mylab{4}{38}{11}&\fourthreeeight{11}&$\surd$& \fanupPIquotient \\\cline{1-4}\cline{6-9}                                                  
\mylab{4}{A}{15}&\fourA{15}&$\surd$& \PIc &&                                           \mylab{4}{38}{15}&\fourthreeeight{15}& $\surd$ & \PIc \\\cline{1-4}\cline{6-9}                                                 
\end{tabular}                                                                                   \end{center}                              


\begin{center}
\begin{tabular}{|c|c|c|c|c|c|c|c|c|}\cline{1-4}\cline{6-9}
\multicolumn{4}{|c|}{\mspk{4}{1F}}&\qquad&\multicolumn{4}{|c|}{\mspk{4}{3E}}\\
\cline{1-4}\cline{6-9}
\mylab{4}{1F}{0}&\fouroneF{0}&$\surd$ &\AFc                                  &&\mylab{4}{3E}{0}&\fourthreeE{0}&$\surd$ &\AFc \\\cline{1-4}\cline{6-9}                          
\mylab{4}{1F}{1}&\fouroneF{1}&$\surd$ &\FININFPIideal                        &&\mylab{4}{3E}{1}&\fourthreeE{1}&$\surd$ & \rref{r:class1}  \\\cline{1-4}\cline{6-9}              
\mylab{4}{1F}{2}&\fouroneF{2}& $\surd_{\unital}$ & \cref{adhocY}                         &&\mylab{4}{3E}{3}&\fourthreeE{3}&$\surd$ & \cref{adhocY} \\\cline{1-4}\cline{6-9}          
\mylab{4}{1F}{3}&\fouroneF{3}&$\surd$ &\FININFPIideal                                  &&\mylab{4}{3E}{4}&\fourthreeE{4}&$\surd_{\fg}$ &\cref{adhocY}  \\\cline{1-4}\cline{6-9}       
\mylab{4}{1F}{4}&\fouroneF{4}&$\surd$ & \rref{r:class2}                        &&\mylab{4}{3E}{5}&\fourthreeE{5}&$\surd_{\fg}$ &\cref{adhocY}  \\\cline{1-4}\cline{6-9}     
\mylab{4}{1F}{5}&\fouroneF{5}&$\surd_{\unital}$ & \cref{adhocY}            &&\mylab{4}{3E}{7}&\fourthreeE{7}&$\surd$& \cref{adhocY}  \\\cline{1-4}\cline{6-9}                             
\mylab{4}{1F}{6}&\fouroneF{6}&$\surd_{\unital}$ & \cref{adhocY}                &&\mylab{4}{3E}{8}&\fourthreeE{8}&$\surd$ &\ETAFideal \\\cline{1-4}\cline{6-9}                           
\mylab{4}{1F}{7}&\fouroneF{7}&$\surd$ &\cref{adhocY}                        &&\mylab{4}{3E}{9}&\fourthreeE{9}&$\surd_{\unital}$ &\cref{adhocY}  \\\cline{1-4}\cline{6-9}               
\mylab{4}{1F}{12}&\fouroneF{12}&$\surd$ &\rref{r:class2}                            &&\mylab{4}{3E}{11}&\fourthreeE{11}&$\surd_{\unital}$ &\cref{adhocY}   \\\cline{1-4}\cline{6-9}   
\mylab{4}{1F}{13}&\fouroneF{13}&$\surd_{\fg}$ &\cref{adhocY}                        &&\mylab{4}{3E}{12}&\fourthreeE{12}&$\surd$ &\FININFAFideal  \\\cline{1-4}\cline{6-9}            
\mylab{4}{1F}{14}&\fouroneF{14}&$\surd$& \cref{adhocY}                    &&\mylab{4}{3E}{13}&\fourthreeE{13}&$\surd$ &\cref{adhocY}  \\\cline{1-4}\cline{6-9}                          
\mylab{4}{1F}{15}&\fouroneF{15}&$\surd$&\PIc                           &&\mylab{4}{3E}{15}&\fourthreeE{15}&$\surd$ & \PIc  \\\cline{1-4}\cline{6-9}                          
\end{tabular}
\end{center}


\begin{center}
\begin{tabular}{|c|c|c|c|}\hline
\multicolumn{4}{|c|}{\mspk{4}{1E}}\\\hline\hline
\mylab{4}{1E}{0}&\fouroneE{0}&$\surd$&\AFc\\\hline          
\mylab{4}{1E}{1}&\fouroneE{1}&$\surd$&\rref{r:class1}\\\hline          
\mylab{4}{1E}{3}&\fouroneE{3}&$\surd$&\rref{r:class1}\\\hline          
\mylab{4}{1E}{4}&\fouroneE{4}&$\surd$&\rref{r:class2}\\\hline          
\mylab{4}{1E}{5}&\fouroneE{5}&?&\\\hline          
\mylab{4}{1E}{7}&\fouroneE{7}&?&\\\hline          
\mylab{4}{1E}{12}&\fouroneE{12}&$\surd$&\rref{r:class2}\\\hline        
\mylab{4}{1E}{13}&\fouroneE{13}&?&\\\hline        
\mylab{4}{1E}{15}&\fouroneE{15}&?&\\\hline        
\end{tabular}
\qquad
\begin{tabular}{|c|c|c|c|}\hline
\multicolumn{4}{|c|}{\mspk{4}{3B}}\\\hline\hline
                    \mylab{4}{3B}{0}&\fourthreeB{0}&$\surd$ & \AFc\\\hline                                 
                     \mylab{4}{3B}{1}&\fourthreeB{1}&$\surd$ &\FININFPIideal\\\hline                             
                     \mylab{4}{3B}{2}&\fourthreeB{2}&?&\\\hline                                           
                     \mylab{4}{3B}{3}&\fourthreeB{3}&?&\\\hline                                          
                     \mylab{4}{3B}{6}&\fourthreeB{6}& ? & \\\hline                                          
                     \mylab{4}{3B}{7}&\fourthreeB{7}&$\surd$ &\cref{adhocO} \\\hline                   
                 \mylab{4}{3B}{8}&\fourthreeB{8}&$\surd$ &\ETAFideal \\\hline                                     
                   \mylab{4}{3B}{9}&\fourthreeB{9}&? & \\\hline                       
                   \mylab{4}{3B}{10}&\fourthreeB{10}&?&\\\hline                                          
                               \mylab{4}{3B}{11}&\fourthreeB{11}&? &\\\hline                             
                                \mylab{4}{3B}{14}&\fourthreeB{14}&$\surd$ &\cref{adhocO} \\\hline 
                                    \mylab{4}{3B}{15} &\fourthreeB{15}&$\surd$ & \PIc   \\\hline    
                                  \end{tabular}
                                \end{center}

\clearpage
              
\providecommand{\bysame}{\leavevmode\hbox to3em{\hrulefill}\thinspace}
\providecommand{\MR}{\relax\ifhmode\unskip\space\fi MR }
\providecommand{\MRhref}[2]{%
  \href{http://www.ams.org/mathscinet-getitem?mr=#1}{#2}
}
\providecommand{\href}[2]{#2}


\end{document}